\documentclass[10pt,reqno]{article}
\usepackage{a4wide}

\usepackage{amssymb}

\usepackage{amssymb, amsfonts, amsbsy, latexsym}
\usepackage{graphicx}
\usepackage{tikz}
\usepackage{color}
\usepackage{colortbl}
\usepackage[english]{babel}
\usepackage{url}
\usepackage{amsmath}

\usepackage{scalerel}

\usepackage{relsize}

\usepackage{amsthm}

\usepackage{bbm}

\usepackage{blkarray}

\usepackage{bm}

\newtheorem{theorem}{Theorem}
\newtheorem{definition}[theorem]{Definition}
\newtheorem{proposition}[theorem]{Proposition}
\newtheorem{remark}[theorem]{Remark}
\newtheorem{lemma}[theorem]{Lemma}
\newtheorem{example}[theorem]{Example}
\newtheorem{assumption}[theorem]{Assumption}

\newtheorem{approximation problem}[theorem]{Approximation problem}

\newtheorem{corollary}[theorem]{Corollary}

\newcommand{\ip}[2]{\left\langle#1,#2\right\rangle}
\newcommand{\abs}[1]{\left|#1\right|}
\newcommand{\norm}[1]{\left\|#1\right\|}

\def\hf{\hat{f}}
\def\hs{\hat{s}}

\def\hq{\hat{q}}

\def\w{\omega}
\def\d{\delta}
\def\e{\epsilon}

\def\k{\kappa}

\def\NN{\mathbb{N}}

\def\cF{\mathcal{F}}

\def\cH{\mathcal{H}}

\def\CC{\mathbb{C}}
\def\NN{\mathbb{N}}

\def\RR{\mathbb{R}}

\begin{document}

\title{Randomized Signal Processing with Continuous Frames\\
$ $
\\
{\large  \ \ \ \ \ \textbf{Ron Levie} \quad \quad \quad \quad \quad \quad \quad\ \quad \ \  \textbf{Haim Avron}} \quad \\
{\normalsize   \quad \quad \textit{levie@math.lmu.de} \ \quad \quad  \quad \quad \ \ \ \ \ \quad \ \ \ \ \textit{haimav@tauex.tau.ac.il} }\\
{\normalsize  \textit{Technische Universit\"{a}t Berlin} \ \ \ \ \  \quad \quad \quad  \quad \quad  \textit{Tel Aviv University} }
}

\date{}

\maketitle

\begin{abstract}
This paper focuses on signal processing tasks in which the signal is transformed from the signal space to a higher dimensional coefficient space (also called phase space) using a continuous frame, processed in the coefficient space, and synthesized to an output signal. 
 We show how to approximate such methods, termed phase space signal processing methods, using a Monte Carlo method.
\textcolor{black}{As opposed to standard discretizations of continuous frames, based on sampling discrete frames from the continuous system, the proposed Monte Carlo method is directly a quadrature approximation of the continuous frame.}
We show that the Monte Carlo method \textcolor{black}{allows working with highly redundant continuous frames,}  since the number of samples required for a certain accuracy is proportional to the dimension of the signal space, and not to the dimension of the phase space. \textcolor{black}{Moreover, even though the continuous frame is highly redundant, the Monte Carlo samples are spread uniformly, and hence represent the coefficient space more faithfully than standard frame discretizations.}
\end{abstract}

\section{Introduction}

\textcolor{black}{We consider signal processing tasks based on continuous frames \cite{Cframe0,Cframe1}. In the general setting, an input signal $s$, from the Hilbert space of signals $\cH$, is first analyzed into a coefficient space representation $V_f[s]$ via the frame analysis operator $V_f$. Then, $V_f[s]$ is manipulated in the coefficient space by first applying a pointwise nonlinearity $r\circ V_f[s]$ and then a linear operator $T$, to produce $T(r\circ V_f[s])$, which is finally synthesized back to the signal space via $V_f^*$.} The end-to-end pipeline reads
\begin{equation}
\label{eq:end-to-end-pssp}
\cH\ni s \mapsto V_f^*T(r\circ V_f[s]).
\end{equation}
We call pipelines of the form (\ref{eq:end-to-end-pssp}) \emph{phase space signal processing}.
The linear operator $T$ models a global change of the signal in the feature space, while the nonlinearity $r$ allows modifying the feature coefficients term-by-term with respect to their values. 

Some examples of continuous frames are the 1D continuous wavelet transform - CWT \cite{Cont_wavelet_original,Ten_lectures}, the short time Fourier transform - STFT \cite{Time_freq}, the Shearlet transform \cite{Shearlet} and the Curvelet transform \cite{Curvelet}. 
\textcolor{black}{Signal processing tasks of the form (\ref{eq:end-to-end-pssp}) are used in a multitude of applications. In multipliers \cite{ex1,ex2,New_mult0,New_mult1,New_mult2}, $T$ is a multiplicative operator and the nonlinearity is trivial $r(x)=x$. Multipliers have applications, for example, in audio analysis
\cite{ex3}
and increasing signal-to-noise ratio
\cite{ex4}. In signal denoising, e.g.,  wavelet shrinkage denoising 
\cite{ex5,ex6},
and Shearlet denoising \cite{ex7},   the linear operator is trivial $T=I$ and $r$ is a nonlinearity that attenuates low values. The same is true in shearlet based image enhancement \cite[Section 4]{ImEnhance0}.
In phase vocoder, $T$ is a dilation operator  and $r$ is a so-called phase correction nonlinearity \cite{phase_vocoder1,phase_vocoder2,vocoder_book,vocoder_imp,New_vocoder0,New_vocoder1,Ottosen2017APV,ltfatnote050}. 
In analysis-based iterative thresholding algorithms for inverse problems, the sparsification step in each iteration can be written as (\ref{eq:end-to-end-pssp}) with $T=I$ and $r$ a thresholding nonlinearity \cite{ItThresh0,SepThresh0}. 
}

\textcolor{black}{We note that the theory presented in this paper works also when the analysis and synthesis in (\ref{eq:end-to-end-pssp}) are done by two different frames sharing the same feature space. An example application is shearlet or curevelet based Radon transform inversion \cite{RadShearletInv,RadShearletInv2,Curve_Radon}, where $T$ is a multiplicative operator, $r$ is thresholding, analysis is done by the curvelet/shearlet frame, and synthesis is done by some modified curvelet/shearlet frame.
For simplicity of the presentation, we stick to the same frame for analysis and synthesis. More accurately, in (\ref{eq:PSSPS22}) and (\ref{eq:PSSPA2}) we extend (\ref{eq:end-to-end-pssp}) to a pipeline based on a frame and its canonical dual.}

\textcolor{black}{In this paper we study quadrature discretizations of (\ref{eq:end-to-end-pssp}) based on random samples.}



\subsection{Quadrature vs. discrete frame discretizations of continuous frames}

An analysis operator of a continuous frame $V_f$ has the form
\begin{equation}
\cH\ni s \mapsto V_f[s] = \ip{s}{f_{(\cdot)}}\in L^2(G).
\label{eq:CHS10}
\end{equation}
where $G$ is a measure space called \emph{phase space}, that usually has some physical interpretation (e.g., in the STFT $G$ is the time-frequency plane), $\{f_g\}_{g\in G}$ is the continuous frame, and $L^2(G)$ is called the coefficient space. Accordingly, the synthesis operator $V_f^*$ has the form \cite[Theorem  2.6]{Cframe1}
\begin{equation}
V_f^*(F) = \int_G F(g)f_gdg.
\label{eq:CHS13}
\end{equation}

As evident from the above description, phase space signal processing involves integrals, and thus some form of discretization is required. 
\textcolor{black}{The common approach is to sample points from $G$ to construct a discrete frame version $\{f_{g^k}\}_{k=1}^{\infty}$ of the continuous frame (e.g., as in \cite{Time_freq,Ten_lectures}). For the discrete frame, the synthesis operator  reads, for $(F_k)_k\in l^2$, 
\begin{equation}
 V_{\{f_{g^k}\}_n}^*\{F_k\}_k = \sum_{k=1}^{\infty} f_{g^k} F_k. 
 \label{DFAO}
\end{equation}
Note that (\ref{DFAO}) looks like a quadrature approximation of (\ref{eq:CHS13}). However, in the standard continuous-to-discrete frame approach, the points $g^k$ are not chosen for approximating (\ref{eq:CHS13}). Rather, $g^k$ are chosen so that the discrete system $\{f_{g^k}\}_{n=1}^{\infty}$ satisfies the discrete frame inequality. Hence, the discrete frame is related to the continuous one by the fact that $f_{g^k}$ are sampled from $f_{(\cdot)}$, not by some approximation requirement. In this paper we take a different route to discretization, requiring that (\ref{DFAO}) approximates (\ref{eq:CHS13}). Let us call the latter discretization approach the \emph{quadrature approach}.}

\textcolor{black}{There is an advantage in the quadrature approach over the discrete frame approach when working with highly redundant continuous frames. To illustrate the idea, consider the STFT, where $G=\RR^2$ is the time-frequency plane. Suppose that we extend $G$ by adding to the time and frequency axes $t,\w$ a third axis $c$ that controls the time width of the window. To discretize the resulting continuous frame $f_{t,\w,c}$ to a discrete frame, one may simply choose to fix the window width axis to one single value $c=c^0$, and sample a time-frequency 2D grid $(t^n,\w^m,c^0)_{n,m}$. Indeed, such an approach would result in a standard discrete STFT, which is a discrete frame. However, the information along the third axis is lost in this discretization. In fact, nothing in the continuous-to-discrete frame approach forces the discretization to faithfully represent the whole continuous phase space. In the quadrature approach, our goal is to discretize the continuous frames more faithfully, sampling uniformly all the feature directions in phase space.}

\subsection{Randomized quadrature approximations of continuous frames}
\label{Randomized quadrature approximations of continuous frames}

\textcolor{black}{Our motivation is hence to discretize highly redundant continuous frames in the quadrature approach. In such situations, the dimensionality of phase space is higher than that of the signal space, and hence using a randomized discretization makes sense. Our approach is motivated by randomized methods in finite-dimensional numerical linear algebra, which are a prominent approach for dealing with high dimensional data~\cite{Mahoney11,Woodruff14,YMM16,DKM06,DM05,CW17}. The goal in this paper is to develop a similar randomized theory in an infinite dimensional setting in general separable Hilbert spaces, namely, in the phase space signal processing setting.}

Randomized algorithms in a context of continuous frames were presented in the past.
\textcolor{black}{\emph{Relevant sampling} is a line of work in which  integral transforms are randomly discretized  \cite{relevantS1,relevantS2,relevantS3,relevantS4}. While the goal in our approach is to approximate the continuous frame with a quadrature sum, the goal in relevant sampling is to sample discrete frames from  continuous frames.
}

We summarize our construction as follows.

\paragraph{Signal processing in phase space.}
Let $V_f:\cH\rightarrow L^2(G)$ be the analysis operator of a continuous frame, and $S_f=V_f^*V_f$ be the frame operator. Since $S_f^{-1}V_f^*V_f=V_f^*V_fS_f^{-1}$ is the identity $I$, we consider the following two formulations of signal processing in phase space. \emph{Synthesis phase space signal processing} is defined by the pipeline
\begin{equation}
s\mapsto V_f^* T r\circ \big(V_f[S_f^{-1}s]\big),
\label{eq:PSSPS22}
\end{equation}
and \emph{analysis phase space signal processing} is defined by
\begin{equation}
s\mapsto S_f^{-1}V_f^* T r\circ \big(V_f[s]\big).
\label{eq:PSSPA2}
\end{equation}
Here, $T$ is a bounded operator in $L^2(G)$ and $r:\CC\rightarrow\CC$ is a nonlinearity. The pipelines (\ref{eq:PSSPS22}) and (\ref{eq:PSSPA2}) can be seen as working with the canonical dual frame  $S_f^{-1}f$ \cite{Cframe1} either in the analysis or in the synthesis step. \textcolor{black}{We suppose that $S_f$ has an efficient discretization in the signal space $\cH$. Hence, we would like to find an efficient discretization of the rest of the pipeline, namely, of $V_f^* T r\circ \big(V_f[s]\big)$.}

\paragraph{Mote Carlo signal processing in phase space.}
We study a Monte Carlo approximation of signal processing in phase space based on the pipelines (\ref{eq:PSSPS22}) and (\ref{eq:PSSPA2}).
We first consider a Monte Carlo approximation of synthesis. For $F\in L^2(G)$, under certain assumptions given in Subsection \ref{Monte Carlo synthesis_}, we consider the approximation of the synthesis operator 
by
\begin{equation}
V_f^*F = \int_G F(g)f_gdg \approx \frac{C}{K}\sum_{k=1}^K F(g^k)f_{g^k}.
\label{eq:PSSPS22MC01}
\end{equation}
Here, $\{g^k\}_{k=1}^K\subset G$ is a finite set of independent random sample points, and $C$ is a normalization constant. 

Using this approximation, in Section \ref{Monte Carlo signal processing in phase space} we study the approximation rate of stochastic signal processing in phase space (\ref{eq:PSSPS22}) and (\ref{eq:PSSPA2}).
The first version of the approximation reads, for the synthesis and analysis  formulations respectively,
\begin{equation}
s\mapsto \frac{C}{K}  \sum_{k=1}^K \big(T(r\circ V_f[S_f^{-1}s])\big)(g^k)f_{g^k},
\label{eq:PSSPA2MC2AA}
\end{equation}
\begin{equation}
s\mapsto \frac{C}{K}  S_f^{-1}\sum_{k=1}^K \big(T(r\circ V_f[s])\big)(g^k)f_{g^k}.
\label{eq:PSSPA2MC2BA}
\end{equation}
Under some general assumptions, we also approximate the signal processing pipelines (\ref{eq:PSSPS22}) and (\ref{eq:PSSPA2}) when $T$ in an integral operator  defined by
\[TF(g) = \int_G R(g,g')F(g')dg',\]
where $R:G^2\rightarrow\CC$ \textcolor{black}{(Definition \ref{phase_space_operator})}. 
The synthesis and analysis approximations in this case read respectively
\begin{equation}
 s\mapsto \frac{C}{KL}  \sum_{k=1}^K\sum_{m=1}^L R(g^k,y^m)r\big(V_f[S_f^{-1}s](f_{y^m})\big)f_{g^k},
\label{eq:PSSPA2MC2AB}
\end{equation}
\begin{equation}
 s\mapsto \frac{C}{KL}  S_f^{-1}\sum_{k=1}^K\sum_{m=1}^L R(g^k,y^m)r\big(V_f[s](f_{y^m})\big)f_{g^k}.
\label{eq:PSSPA2MC2BB}
\end{equation}
Here, $\{g^k\}_{k=1}^K,\{y^m\}_{m=1}^L\subset G$  are two  finite sets of independent random sample points.

Methods (\ref{eq:PSSPA2MC2AB}) and (\ref{eq:PSSPA2MC2BB}) are useful for integral operators. Methods (\ref{eq:PSSPA2MC2AA}) and (\ref{eq:PSSPA2MC2BA}) are useful when the samples $\big(Tr\circ V_f[s]\big)(g^k)$ can be computed using some other samples $V_f[s](y^k)$ of $V_f[s]$,  which is the case for \textcolor{black}{multiplicative} and diffeomorphism operators (see Subsection \ref{Stochastic diffeomorphism operator}).

\paragraph{Summary of our main results.}
\begin{itemize}
	\item
	We prove the convergence of the Monte Carlo methods (\ref{eq:PSSPA2MC2AA})--(\ref{eq:PSSPA2MC2BB}) as the number as samples increase, and also introduce non-asymptotic error bounds. 
	When considering discrete signals of resolution/dimension $M$, embedded in the continuous signal space,
	the error in the stochastic method is of order $O(\sqrt{M/K})$, where $K$ is the number of Monte Carlo samples.
	\item
	\textcolor{black}{The computational complexity of our method does not depend on the dimension of phase space for a rich class of signal processing pipelines. This allows approximating highly redundant continuous frames efficiently  using sample points which are well spread in all directions in phase space.}
	\item
	\textcolor{black}{As a toy application of the theory, we show how to increase the expressive capacity of the 2D time-frequency phase space by adding a third axis, and use the construction in a phase vocoder scheme.}
\end{itemize}

\textcolor{black}{The proofs in this paper are inspired by the constructions in standard Monte Carlo theory (see, e.g., \cite[Section 2]{QMC0}), adapted to infinite dimensional Hilbert spaces.}

\section{Background: harmonic analysis in phase space}

In this section we review the theory of continuous frames, and give the two examples: STFT and CWT. 
 By convention, all Hilbert spaces in this paper are assumed to be separable. 
The Fourier transform $\cF$ is defined 
 with the following normalization
\begin{equation}
[\cF s](\w) =\hat{s}(\w)= \int_{\RR}s(t)e^{-2\pi i \w t}dt, \quad [\cF^{-1}\hs](t) = \int_{\RR}\hs(\w)e^{2\pi i \w t}d\w.  
\label{eq:Fourier_trans}
\end{equation} 
We denote the norm of a vector $v$ in a Banach space $\mathcal{B}$ by $\norm{v}_{\mathcal{B}}$. For a measure space $\{G,\mu\}$, we denote interchangeably by 
\[\norm{f}_p= \norm{f}_{L^p(G)} = \Big( \int_G \abs{f(g)}^p d\mu(g) \Big)^{1/p}\]
the $p$ norm of the signal $f\in L^p(G)$, where $1\leq p < \infty$, and denote interchangeably
\[\norm{f}_{\infty} = \norm{f}_{L^{\infty}(G)}= {\rm ess}\ {\rm sup}_{g\in G}\abs{f(g)}\] 
for $f\in L^{\infty}(G)$. \textcolor{black}{We note that an equality between two $L^p$ functions is by definition an almost-everywhere equality.}  \textcolor{black}{We denote the induced operator norm of operators in Banach spaces using the subscript of the Banach space, e.g., for a  bounded linear operator $T:L^p(G)\rightarrow L^p(G)$, we denote $\norm{T}_{p}$. When we want to emphasize that the norm is an operator norm, we also denote $\norm{T}_{p\rightarrow p}$.} 

\subsection{Continuous frames}


The following definitions and claims are from \cite{Cframe1} and \cite[Chapter 2.2]{Fuhr_wavelet}, with notation adopted from the later.

\begin{definition}
\label{CSSframe}
Let $\cH$ be a Hilbert space, and $(G,\mathcal{B},\mu)$ a locally compact topological space with Borel sets $\mathcal{B}$, and $\sigma$-finite Borel measure $\mu$.
 Let $f:G\rightarrow \cH$ be a weakly measurable mapping, namely for every $s\in\cH$
\[g\mapsto \ip{s}{f_g}\]
is a measurable function $G\rightarrow\CC$.
For any $s\in\cH$, we define the \emph{coefficient function}
\begin{equation}
V_f[s]:G\rightarrow \CC \quad , \quad V_f[s](g)=\ip{s}{f_g}_{\cH}.
\label{eq:CSS2frame}
\end{equation}
\begin{enumerate}
	\item
	We call $f$ a \emph{continuous frame}, if $V_f[s]\in L^2(G)$ for every $s\in\cH$, and there exist constants $0<A\leq B<\infty$ such that
	\begin{equation}
	A\norm{s}_{\cH}^2 \leq \norm{V_f[s]}_2^2 \leq B\norm{s}_{\cH}^2
	\label{eq:FB}
	\end{equation}
	for every $s\in\cH$.
	\item
	If it is possible to choose $A=B$, $f$ is called a \emph{tight frame}.
	\item
	We call $\cH$ the \emph{signal space},
	 $G$ \emph{phase space}, $V_f$ the \emph{analysis operator}, and $V_f^*$ the \emph{synthesis operator}.
	\item
	We call the frame $f$ \emph{bounded}, if there exist a constant $0<C\in\RR$ such that
\[\forall g\in G\ , \norm{f_g}_{\cH}\leq C.\]
\item
We call $S_f=V_f^*V_f$ the \emph{frame operator}, and $Q_f=V_f V_f^*$ the \emph{Gramian operator}.
\item
	\textcolor{black}{We call $f$ a \emph{Parseval} continuous frame, if $V_f$ is an isometry between $\cH$ and $L^2(G)$.}
\end{enumerate}
\end{definition}

\begin{remark}
\textcolor{black}{A frame is Parseval if and only if the frame bounds cane be chosen as $A=B=1$.}
\end{remark}

For the  closed form formula of the synthesis operator $V_f^*$, we recall the notion of weak vector integrals, also called Pettis integral,  introduced in \cite{Weak_Integral}.

\begin{definition}
\label{D:week_vec_int}
Let $\cH$ be a separable Hilbert space, and $G$ a measure space. Let $v:G\rightarrow \cH$ be a mapping such that the mapping
 $s\mapsto\int_G \ip{s}{v(g)}dg$ 
is continuous in $s\in\cH$. Then the \emph{weak vector integral} (or weak $\cH$ integral) is defined to be the vector 
$\int^{\rm w}_G v(g)dg \in \cH$ 
such that
\[\forall s\in\cH \  \quad \int_G \ip{s}{v(g)}dg=\ip{s}{\int^{\rm w}_G v(g)dg}.\]
The existence of such a vector is guaranteed by Riesz representation theorem. In this case, $v$ is called a \emph{weakly integrable function}.
\end{definition}

Given a continuous frame, the synthesis operator can be written by \cite[Theorem  2.6]{Cframe1}
\begin{equation}
V_f^*[F] = \int^{\rm w}_G F(g)f_g dg.
\label{eq:inver_proj}
\end{equation}

\begin{definition}
The \emph{frame kernel} $K_f:G^2\rightarrow\CC$ is defined by
\begin{equation}
K_f(g,g')=\ip{f_g}{f_{g'}} = V_f[f_g](g').
\label{eq:a4}
\end{equation}
\end{definition}

The following result is taken from \cite[Proposition 2.12]{Fuhr_wavelet}.
\begin{proposition}
\label{Gram_Ker}
The Gramian operator $Q_f$ is an integral operator with kernel $K$. Namely, for every $F\in L^2(G)$
\begin{equation}
[Q_fF](g')  = \int_G F(g)  K_f(g,g')dg.
\label{eq:a3}
\end{equation}
\end{proposition}





For a Parseval frame,
the image space $V_f[\cH]$ is a reproducing kernel Hilbert space, \textcolor{black}{with kernel} $K_f(g,\cdot)$, and the orthogonal projection upon $V_f[\cH]$ is given by the Gramian operator $Q_f=V_fV_f^*$.




\subsection{Examples}

\textcolor{black}{An important class of Parseval frames are wavelet transforms based on square integrable representations, which we call in this paper simply wavelet transforms. We refer the reader to  \cite[Chapters 2.3--2.5]{Fuhr_wavelet}, and the classical papers \cite{gmp0,gmp}.} 
The wavelet system in the general theory is generated by fixing one signal $f\in \cH$, that is typically called the \emph{mother wavelet} or the \emph{window function}, and applying on it a set of transformations $\{\pi(g)f\ |\ g\in G\}$, parameterized by a locally compact topological group $G$. The Haar measure is taken in $G$, and $\pi:G\rightarrow {\cal U}(\cH)$ is assumed to be a \emph{square integrable representation},  where ${\cal U}(\cH)$ is the group of unitary operators in $\cH$.  

The wavelet transform is defined by
\[V_f:\cH\rightarrow L^2(G) \quad , \quad V_f[s](g)=\ip{s}{\pi(g)f}.\]
For any two mother wavelets $f_1$ and $f_2$, the reconstruction formula of the wavelet transform is given by
\[s=\frac{1}{\ip{Af_2}{Af_1}}V_{f_2}^*V_{f_1}(s) =\frac{1}{\ip{Af_2}{Af_1}}\int^{\rm w}_G V_{f_1}[s](g)\pi(g)f_2 \ dg.\]
Here, $A$ is a special positive operator in $\cH$, called the \emph{Duflo-Moore operator}, uniquely defined for every square integrable representation $\pi$, that determines the normalization of windows.

\paragraph{The short time Fourier transform.} 
The following construction is taken from \cite{Time_freq}.
Consider the signal space $L^2(\RR)$. Let $\mathcal{T}:\RR\rightarrow {\cal U}(L^2(\RR))$ be the translation in $L^2(\RR)$, defined for $x\in\RR$ and $f\in L^2(\RR)$ by $[\mathcal{T}(x)f](t)=f(t-x)$. \textcolor{black}{Let $\mathcal{M}:\RR\rightarrow {\cal U}(L^2(\RR))$ be the modulation} in $L^2(\RR)$, defined for $\w\in\RR$ and $f\in L^2(\RR)$ by $[\mathcal{M}(\w)f](t)=e^{2  \pi i \w t}f(t)$. 
%
Denote $\pi(x,\w)= \mathcal{T}(x)\mathcal{M}(\w)$.  
For a normalized window $f$, the mapping
\[\RR^2\ni(x,\w)\mapsto \pi(x,\w)f\]
is a Parseval continuous frame, with the standard Lebesgue measure of the phase space $\RR^2$. 
The resulting transform $V_f[s](x,\w)=\ip{s}{\pi(x,\w)f}$ is called the \emph{Short Time Fourier Transform} (STFT).

\subsubsection{The 1D continuous wavelet transform}


The following construction is taken from \cite{Cont_wavelet_original,Ten_lectures}.
Consider the signal space $L^2(\RR)$, and the translation $\mathcal{T}$ as in the STFT. \textcolor{black}{Let $\mathcal{D}:\RR\setminus\{0\}\rightarrow {\cal U}(L^2(\RR))$ be the dilation in $L^2(\RR)$, defined for $\tau\in\RR\setminus\{0\}$ and $f\in L^2(\RR)$ by $[\mathcal{D}(\tau)f](t)=\frac{1}{\sqrt{\abs{\tau}}}f(\frac{t}{\tau})$.} The set of transformations
\begin{equation}
{\cal A}=\{ \mathcal{T}(x)\mathcal{D}(\tau)\ |\ (x,\tau)\in\RR\times (\RR\setminus\{0\})\}
\label{eq:Affine_group1}
\end{equation}
 is closed under compositions. We can treat ${\cal A}$ as a group of tuples $\RR\times (\RR\setminus\{0\})$, with group product derived from the compositions of operators in (\ref{eq:Affine_group1}). The group ${\cal A}$ is called the 1D \emph{affine group}. The mapping
\[\pi(x,\tau)=\mathcal{T}(x)\mathcal{D}(\tau)\]
 is a square integrable representation, with Dulfo-Moore operator $A$ defined by 
$
[\cF A \cF^*\hf](z) = \frac{1}{\sqrt{\abs{z}}}\hf(z)
$, 
where $\cF$ is the Fourier transform.
 The resulting wavelet transform is called the \emph{Continuous Wavelet Transform} (CWT).

Next, we show how the CWT atoms are interpreted as time-frequency atoms, and the CWT is interpreted as a time-frequency transform.
Here, by changing variable $\w=\frac{1}{\tau}$, we obtain the Parseval frame 
\[\{\pi'(x,\w)f\}_{(x,\w)\in\RR\times (\RR\setminus\{0\})}\]
based on the representation 
$\pi'(x,\w)=\mathcal{T}(x)\mathcal{D}(\w^{-1})$. 
The parameter $\w$ is interpreted as frequency. The mapping $\pi'$ is a representation of the 1D affine group with the new parameterization $\w=\frac{1}{\tau}$, in which the Haar measure is the standard Lebesgue measure of $\RR\times (\RR\setminus\{0\})$.

\section{Elements of stochastic signal processing in phase space}
\label{Phase space Monte Carlo method for coherent state systems}

In this section we develop basic approximation results that will be used later in the paper to  bound the approximation error between the signal processing pipelines (\ref{eq:PSSPS22}) and (\ref{eq:PSSPA2}), and their   stochastic approximations (\ref{eq:PSSPS22MC01})--(\ref{eq:PSSPA2MC2BB}). 


\subsection{Phase space operators}

We start by defining integral operators in the coefficient space. 
%
%

\textcolor{black}{
\begin{definition}
\label{phase_space_operator}
Let $T$ be a bounded linear operator in $L^2(G)$, where $G$ is a locally compact topological space with $\sigma$-finite Borel measures.
\begin{enumerate}
	\item
	We call $T$ a \emph{phase space integral operator (PSI operator)} if there exists a measurable function $R:G\times G\rightarrow\CC$ 
	with $R(\cdot,g)\in L^2(G)$ for almost every $g\in G$, such that for every $F\in L^2(G)$
\begin{equation}
TF = \int_{G} R(\cdot,g)F(g)dg.
\label{eq:PSO}
\end{equation}
\item
 A phase space integral operator $T$ is called \emph{uniformly square integrable}, if there is a constant $D>0$ such that for almost every $g\in G$
 \begin{equation}
 \label{eq:USI}
  \norm{R(\cdot,g)}_{L^2(G)}=\sqrt{\int_{G}  \abs{R(g',g)}^2 dg'} \leq D.   
 \end{equation}
\end{enumerate}
\end{definition}
}

\textcolor{black}{
\begin{example}
\label{Ex:Qf_bound}
The Gramian operator $Q_f$ of a continuous frame is a phase space operator by Proposition \ref{Gram_Ker}, with $\norm{Q_f}_2\leq B$. If $f$ is bounded, with bound $\norm{f_g}_{\cH}\leq C$, then  $Q_f$ is uniformly square integrable with bound
\[\norm{K_f(g,\cdot)}_{2} = \sqrt{\int_{G}  \abs{K_f(g,g')}^2 dg'} =  \sqrt{\int_{G}  \abs{V_f[f_g](g')}^2 dg'} \leq B^{1/2}C. \]
\end{example}
}


\subsection{Sampling in phase space}
\label{Sampling in phase space}

Let $F\in L^2(G)$, and let $f$ be a continuous frame.
The phase space $G$ in general \textcolor{black}{does not have finite measure}, and thus uniform sampling is not defined on $G$. However, when $G$ has infinite measure, functions $F\in L^2(G)$ must decay in some sense ``at infinity'', so it is possible  to restrict our sampling to a compact domain in  $G$, in which $F$ has most of its energy. 
More accurately, since $G$ is $\sigma$-finite, it is the disjoint union of at most countably many sets of finite measure. 
 Namely, there are disjoint measurable sets $X_n$ of finite measure, with $\bigcup_{n\in\NN}X_n= G$, such that for every $F\in L^2(G)$
\begin{equation}
\norm{F}_{L^2(G)}^2 = \int_G \abs{F(g)}^2dg = \sum_{n\in\NN}\int_{X_n} \abs{F(g)}^2dg=\sum_{n\in\NN}\norm{F}_{L^2(X_n)}^2.
\label{eq:sigint0}
\end{equation}
Denote $G_n=\bigcup_{j=1}^n X_n$, and note that $\bigcup_{n\in\NN}G_n= G$. Now, (\ref{eq:sigint0}) is equivalent to
\[\norm{F}_{L^2(G)}^2 = \lim_{n\rightarrow\infty}\norm{F}_{L^2(G_n)}^2.\]
Thus, for every $\e>0$, there exists an indicator function $\psi_{\e}$ (that depends on $F$) of a measurable set of finite measure $\norm{\psi_{\e}}_1$,  such that
\begin{equation}
\norm{\psi_{\e} F - F}_2<\e.
\label{eq:env000}
\end{equation}
In our analysis, we allow more general forms of \emph{envelopes} $\psi_{\epsilon}$.
\textcolor{black}{
\begin{definition}
\label{def:env}
An \emph{envelope} is a positive $\psi\in L^1(G)\cap L^{\infty}(G)$   satisfying  $\norm{\psi}_{\infty}\leq 1$. 
\end{definition}
}
Given an envelope $\psi$, samples can be drawn from $G$ according to the probability  density $\frac{\psi(g)}{\norm{\psi}_1}$. 
In the following analysis we fix an envelope $\psi_{\e}=\psi$ independently of a specific function $F\in L^2(G)$. This is the common approach in classical signal processing, where a compact frequency band $[a,b]$ is predefined independently of a specific signal. It is implicit that we can only treat signals having most of their frequency energy in $[a,b]$. Any frequency information outside of $[a,b]$ is lost or projected into the band. 
\textcolor{black}{In Section \ref{Discrete stochastic phase space signal processing}, and specifically Definition \ref{D:linear area discretizable}, we study the support of $\psi$ required to capture most of the energy of discrete signals.}


\subsection{Input sampling in phase space operators}
\label{Monte Carlo method of phase space functions}

Given a PSI operator $T$ with kernel $R$, in this subsection we sample the input variable $g$ of $R(g',g)$, and keep the output variable $g'$ continuous. In Subsection \ref{Monte Carlo signal processing in phase space} we show that sampling the output variable $g'$ is a special case of the framework developed in this subsection.
Let $\psi$ be an envelope, and $g\in G$ be a random sample according to the probability distribution $\frac{\psi(g)}{\norm{\psi}_1}$. Define the random rank one operator \textcolor{black}{$T^{\psi,1}$}, applied on $F\in L^2(G)$, by
\[g\mapsto (T^{\psi,1}F)(g)=\norm{\psi}_1 R(\cdot,g) F(g).\]
We also denote 
$T^{\psi,1}F(g';g)=\norm{\psi}_1 R(g',g) F(g)$, 
where $g'$ is the variable of the output function $T^{\psi,1}F$. 
Next, we define the Monte Carlo approximation of $TF$ as a sum of independent $T^{\psi,1}F$ vectors. 

\begin{definition}[Input Monte Carlo phase space operator]
\label{Input Monte Carlo phase space operator}
Let $T$ be a PSI operator in $L^2(G)$ (Definition \ref{phase_space_operator}), \textcolor{black}{$\psi$ an envelope,} $F\in L^2(G)$, and $K\in\NN$. Let $G_k=G$, $k=1,\ldots,K$, be $K$ copies of  $G$,  and let $\{g^k\}_{k=1}^K$ denote a random sample from $G_1\times\ldots\times G_K$ with the probability distribution $\prod_{k=1}^K\frac{\psi(g^k)}{\norm{\psi}_1}$. Let $T^{\psi,1}_kF: G_k\rightarrow L^2(G)$ be
 the random vectors  defined for $g^k\in G_k$ by $[T^{\psi,1}_kF](g^k) =[T^{\psi,1}F](g^k)$, $k=1,\ldots, K$.
Define the random vector $T^{\psi,K}F: G_1\times\ldots\times G_K \rightarrow L^2(G)$ by
\begin{equation}
    \label{TEMP012}
    [T^{\psi,K}F](g';g^1,\ldots,g^K):=\frac{1}{K}\sum_{k=1}^K [T^{\psi,1}_kF](g';g^k) = \frac{\norm{\psi}_1}{K}\sum_{k=1}^K  R(g', g^k) F(g^k).
\end{equation}
We call $T^{\psi,K}F$ the \emph{Monte Carlo phase space integral operator} applied on $F$ and based on $K$ samples, approximating $TF$.
\end{definition}
\textcolor{black}{When the envelope $\psi$ is fixed throughout the analysis, we often denote interchangeably $T^KF=T^{\psi,K}F$. In the following we fix an envelope $\psi$.}
\textcolor{black}{
\begin{remark}
Note that $T^kF$ is a random variable -- a function with $(\{g^k\}_{k=1}^K,g')$  as the variable. Thus, we can sample $L^2(G)$ vectors in (\ref{TEMP012}), which are equivalence classes of functions, without requiring any continuity assumption on 
$R(\cdot,\cdot\cdot)F(\cdot\cdot)$. Indeed, $T^kF$ is defined up to a set of tuples $(\{g^k\}_{k=1}^K,g')$ of measure zero.
\end{remark}
}

The expected value $\mathbb{E}(T^K F)\in L^2(G)$ of $T^K F$ is a function in  $L^2(G)$, defined by
\textcolor{black}{
\[\mathbb{E}(T^K F) = \int_{G^K} [T^KF]\big((\cdot);g^1,\ldots,g^K\big) \frac{\psi(g^1)}{\norm{\psi}_1}dg^1\ldots \frac{\psi(g^K)}{\norm{\psi}_1}dg^K.\]
}
  \textcolor{black}{We define the variance $\mathbb{V}(T^K F)$ as the  integral 
	\[\mathbb{V}(T^K F) = \int_{G^K} \abs{[T^KF]\big((\cdot);g^1,\ldots,g^K\big)-\mathbb{E}^{\rm w}(T^K F)}^2 \frac{\psi(g^1)}{\norm{\psi}_1}dg^1\ldots \frac{\psi(g^K)}{\norm{\psi}_1}dg^K.\]
	}
\textcolor{black}{Given an envelope $\psi$, by abuse of notation, we also denote by $\psi$ the multiplicative operator
\[\psi: L^2(G)\rightarrow L^2(G) \quad , \quad [\psi F](g)=\psi(g)F(g).\]
}


\begin{proposition}
\label{Var_approx1}
Let $T$ be a PSI operator, and $F\in L^2(G)$. Then,
\begin{enumerate}
	\item the expected value $\mathbb{E}(T^KF)$ is in $L^2(G)$ and satisfies
	\begin{equation}
	\mathbb{E}(T^KF)=T(\psi F),
	\label{eq:pointEK}
	\end{equation}
	\item
	if  $T$ is a uniformly square integrable PSI operator, with bound $D$, then 
\begin{equation}
\mathbb{V}(T^KF)=\frac{1}{K}\mathbb{V}(T^1F) \in L^{1}(G),
\label{eq:Var_approx1}
\end{equation}
and
$\norm{\mathbb{V}(T^KF)}_1 \leq \frac{1}{K}\norm{\psi}_1D^2\norm{F}_2^2$.
\end{enumerate}
\end{proposition}

The proof of \textcolor{black}{Item} 1 of Proposition \ref{Var_approx1} follows directly from the definition of PSI operators (Definition \ref{phase_space_operator}). Indeed, for $T^1F$,
\[\mathbb{E}(T^1F)(g') = \int_{G} \norm{\psi}_1 R(g',g) F(g) \frac{\psi(g)}{\norm{\psi}_1}dg = T(\psi F)(g'),\]
and for $T^K$ we use linearity.
The proof of Item 2 of Proposition \ref{Var_approx1} is in the next subsection.  
We next bound the average square error in approximating $T[\psi F]$ by $T^kF$. 

\begin{proposition}
\label{Input_proposition}
Let $f$ be a continuous frame, and $T$ a uniformly square integrable PSI operator with bound $D$. Then
	\[\mathbb{E}\Big(\norm{T^KF-T(\psi F)}_2^2 \Big)\leq  \frac{\norm{\psi}_1}{K} D^2\norm{F}^2_2.\]
\end{proposition}

\begin{proof}
By the Fubini–Tonelli theorem,  and Proposition \ref{Var_approx1}, 
\[
\begin{split}
 & \mathbb{E}\Big( \norm{T^KF-T(\psi F)}_2^2 \Big) \\
& =\int_{G} \int_G \abs{[T^KF](g';g^1,\ldots,g^K)-T(\psi F)(g')}^2 \\
& \quad\quad\quad\quad\quad\quad \quad\quad\quad\quad\quad\quad \quad\quad\quad  dg' \psi(g^1)/\|\psi\|^1_1\cdots\psi(g^K)/\|\psi\|^K_1 dg^1\ldots dg^K \\
& =\norm{\mathbb{V}(T^kF)(g')}_1 \leq  \frac{\norm{\psi}_1}{K} D^2\norm{F}^2_2.
\end{split}
\]

\end{proof}
The \textcolor{black}{expected error} in Proposition \ref{Input_proposition} is
  pointwise in $F$. We note that \textcolor{black}{an operator expected error bound} of the form ``$\mathbb{E}\norm{T^K-T\big(\psi (\cdot)\big)}^2_{2}= O(\frac{\norm{\psi}_1}{K})$''
 like in the finite dimensional matrix operator case \cite{MatrixConcentration} is not possible. Indeed, for any sample set $\{g^k\}_{k=1}^K$ there is a normalized function $F\in L^2(G)$ supported in $G\setminus \{g^k\}_{k=1}^K$, so $T^KF=0$ and $\norm{T^K-T\big(\psi (\cdot)\big)}_2 \geq \norm{T^KF-T(\psi F)}_2=\norm{T(\psi F)}_2$. We thus focus in this paper on pointwise error estimates.
  
The following is an important special case of Propositions \ref{Var_approx1} and \ref{Input_proposition}.
\textcolor{black}{
\begin{corollary}
\label{Ex_Q_E}
Let $T=Q_f=V_f V^*_f$ be the Gramian operator of a bounded continuous frame with $\norm{f_g}_{\cH}\leq C$ and upper frame bound $B$. Let $K_f$ be the frame kernel. Then, we have
\begin{equation}
Q_f^K F= \frac{\norm{\psi}_1}{K} \sum_{k=1}^K F(g^k)K_f(g^k,\cdot),
\label{eq:R1P}
\end{equation}
and
\[\mathbb{E}(Q_f^K F) = Q_f (\psi F).\]
Moreover, $Q_f$ is a uniformly square integrable PSI operator with bound $B^{1/2}C$, 
 and
\[\mathbb{E}\Big(\norm{Q_f^KF-Q_f(\psi F)}_2^2 \Big)\leq  \frac{\norm{\psi}_1}{K} BC^2\norm{F}^2_2.\]
\end{corollary}
\begin{proof}
By Example \ref{Ex:Qf_bound}, $Q_f$ is a uniformly square integrable PSI operator with bound $B^{1/2}C$.
The rest of the results follow from Propositions \ref{Var_approx1} and \ref{Input_proposition}. 
\end{proof}
}

\subsection{Proof of Proposition \ref{Var_approx1}}
\label{Proofs of Subsection a}

The proof is based on the following lemma.

\begin{lemma}
\label{E_var_monte0}
Let $T$ be a PSI operator with kernel $R(g',g)$, let $\psi$ be an envelope, and let $F\in L^2(G)$. Then the following holds.
\begin{enumerate}
	\item
	The expected value of $T^1F$  satisfies
	\begin{equation}
	\mathbb{E}(T^1F)=T(\psi F).
	\label{eq:pointE0}
	\end{equation}
\item
If $T$ is a uniformly square integrable PSI operator,  then $\mathbb{V}(T^1F)\in L^1(G)$, and
\begin{equation}
\mathbb{V}(T^1F)= \mathbb{E}\Big(\abs{T^1 F}^2\Big) - \abs{\mathbb{E}(T^1F)}^2.
\label{eq:E_var_monte02}
\end{equation}
Here, $\abs{T^1 F}^2$ is the function $(g;g')\mapsto\abs{\norm{\psi}_1 R(g',g) F(g)}^2$, and expected value is with respect to the random variable $g$.
\item
If $T$ is a uniformly square integrable PSI operator, with bound $D$,  then 
\begin{equation}
\label{eq:prop14_3}
    \norm{\mathbb{V}(T^1F)}_1 \leq \norm{\psi}_1D^2\norm{F}_2^2.
\end{equation}
\end{enumerate}
\end{lemma}

\begin{proof}
\textcolor{black}{Part 1 was shown in the discussion below Proposition \ref{Var_approx1}. For parts 2 and 3, we write
\begin{equation}
   \begin{split}
  \mathbb{V}(T^1F)(g') =&   
  \int_{G} \abs{[T^1F]\big(g';g\big)}^2 \frac{\psi(g)}{\norm{\psi}_1}dg \\
  & -2{\rm Real} \int_{G} [T^1F]\big(g';g\big)\overline{\mathbb{E}(T^1 F)(g')} \frac{\psi(g)}{\norm{\psi}_1}dg \\
  & + \int_{G} \abs{\mathbb{E}(T^1 F)(g')}^2 \frac{\psi(g)}{\norm{\psi}_1}dg.
\end{split} 
\label{eq:p1}
\end{equation}
We first use the Fubini–Tonelli theorem to prove integrability with respect to $g'$ of the first term of (\ref{eq:p1}). By the fact that $T$ is uniformly square integrable  (Definition \ref{phase_space_operator}.2), and by $\norm{\psi}_{\infty}\leq 1$,
\begin{equation}
   \begin{split}
  \int_{G}\int_{G} \abs{[T^1F]\big(g';g\big)}^2 dg'\frac{\psi(g)}{\norm{\psi}_1} dg & =\norm{\psi}_1 \int_{G}\int_{G} \abs{R(g',g)}^2 dg' \abs{F(g)}^2\psi(g) dg\\
  & \leq \norm{\psi}_1 D^2\norm{F}_2^2,
\end{split} 
\label{eq:temp345d}
\end{equation}
so $\int_{G} \abs{[T^1F]\big(g';g\big)}^2 \frac{\psi(g)}{\norm{\psi}_1}dg \in L^1(G)$ with respect to $g'$.
For the second term of (\ref{eq:p1}), we have
\[
\begin{split}
&\int_{G}\int_{G} [T^1F]\big(g';g\big)\overline{\mathbb{E}(T^1 F)(g')} \frac{\psi(g)}{\norm{\psi}_1}dg dg'\\
& = \int_{G}\int_{G}  R(g',g)F(g) \psi(g) dg  \overline{T(\psi F)(g')}dg' = \norm{\mathbb{E}(T^1 F)}_2^2.
\end{split}
\]
This leads to (\ref{eq:E_var_monte02}).
}
\textcolor{black}{
By the non-negativity of the integrand in the definition of $\mathbb{V}(T^1F)(g')$, we can write 
$\norm{ \mathbb{V}(T^1F)}_1 = \int_{G}\mathbb{V}(T^1F)(g') dg'$,
so by (\ref{eq:E_var_monte02}) and (\ref{eq:temp345d}), we get (\ref{eq:prop14_3}).
}

\end{proof}

\begin{proof}[Proof of Proposition \ref{Var_approx1}]
Part 1 was shown in the discussion below Proposition \ref{Var_approx1}. Next we show Part 2.
By the Fubini–Tonelli theorem, we have
\begin{equation}
\begin{split}
& \norm{\mathbb{V}(T^KF)}_1 = \\
& \int_{G^K}\int_{G} \abs{\frac{1}{K}\sum_{k=1}^K [T^1_kF](g') - \mathbb{E}(T^KF)(g')}^2    dg'  \norm{\psi}_1^{-K}\psi(g^1)dg^1 \psi(g^K)dg^K
\end{split}
\label{eqinproof6r}
\end{equation}
When expanding the product in (\ref{eqinproof6r}),
we have the term
\[
\int_{G^K}\int_{G} \frac{1}{K^2}\sum_{k=1}^K \abs{[T^1_kF](g') - \mathbb{E}(T^KF)(g')}^2  dg'\norm{\psi}_1^{-K}\psi(g^1)dg^1\ldots \psi(g^K)dg^K, 
\]
and mixed terms, for $k\neq k'$,
\[
\begin{split}
 & \int_{G^K}\int_{G} \frac{1}{K^2}\Big([T^1_kF](g^k;g') - \mathbb{E}(T^KF)(g')\Big) \times  \\
 & \quad \quad 
 \overline{\Big([T^1_{k'}F](g^{k'};g') - \mathbb{E}(T^KF)(g')\Big)} dg' \norm{\psi}_1^{-K}\psi(g^1)dg^1\ldots \psi(g^K)dg^K \\
 &  =\int_{G}\int_{G} \overline{\Big([T^1_kF](g^{k'};g') - \mathbb{E}(T^KF)(g')\Big)} \frac{1}{K^2}  \times \\
 &  \quad \quad \quad
    \int_G\Big([T^1_kF](g^k;g') - \mathbb{E}(T^KF)(g')\Big) \norm{\psi}_1^{-1}\psi(g^k)dg^{k} \ dg'\  \norm{\psi}_1^{-1}\psi(g^{k'})dg^{k'},
\end{split}
\]
which are equal to zero, 
since
\[
\begin{split}
 & \int_{G}\Big([T^1_kF](g^k;g') - \mathbb{E}(T^KF)(g')\Big) \norm{\psi}_1^{-1}\psi(g^k)dg^{k} \\
 & =\int_{G}[T^1_kF](g^k;g')\norm{\psi}_1^{-1}\psi(g^k)dg^{k} - \mathbb{E}(T^1_k F)(g') = 0.
\end{split}
\]
Here, the fact that $\overline{\Big([T^1_{k'}F](g^{k'};(\cdot)) - \mathbb{E}(T^KF)(\cdot)\Big)}\in L^2(G)$ for a.e. $g^{k'}$,
and the fact that $\norm{\psi}_1^{-1}\psi(g^k)dg^k$ and $\norm{\psi}_1^{-1}\psi(g^{k'})dg^{k'}$ are probability measures,  
justify the above use of Fubini's theorem.

We thus have, by part 3 of Lemma \ref{E_var_monte0}, 
\[
\begin{split}
\norm{\mathbb{V}(T^KF)}_1 & =\int_{G^K}\int_{G} \frac{1}{K^2}\sum_{k=1}^K\abs{ \Big([T^1_kF](g^k;g') - \mathbb{E}(T^KF)(g')\Big)}^2 \\ 
 & \quad\quad\quad\quad\quad\quad\quad\quad\quad\quad\quad\quad dg'  \norm{\psi}_1^{-K}\psi(g^1)dg^1\ldots \psi(g^K)dg^K \\
 = & \frac{1}{K^2}\sum_{k=1}^K\int_{g^k}\int_{g'} \abs{ \Big([T^1_kF](g^k;g') - \mathbb{E}(T^KF)(g')\Big)}^2  dg'  \norm{\psi}_1^{-1}\psi(g^k)dg^k \\
 = &  \frac{1}{K^2}\sum_{k=1}^K\mathbb{V}(T^1_kF)  =  \frac{1}{K}\mathbb{V}(T^1_kF) \leq \frac{1}{K}\norm{\psi}_1D^2\norm{F}_2^2. 
\end{split}
\]

\end{proof}

\subsection{Monte Carlo synthesis}
\label{Monte Carlo synthesis_}

In this subsection, we use the results of Subsection \ref{Monte Carlo method of phase space functions} to define and analyze the Monte Carlo approximation of synthesis. 
\textcolor{black}{
\begin{definition}
Let $\psi$ be an envelope and $\{g^k\}_{k=1}^K$ random samples  as in Definition \ref{Input Monte Carlo phase space operator}.
Given $F\in L^2(G)$, the \emph{Monte Carlo synthesis}  $V_f^{*\psi,K}F$ is the random variable $G^K\rightarrow \cH$ defined as
\[V_f^{*\psi,K} F =\frac{\norm{\psi}_1}{K}\sum_{k=1}^K F(g^k)f_{g^k}.\]
\end{definition}
When the envelope $\psi$ is constant in the analysis, we often denote the Monte Carlo synthesis in short by $V_f^{* K}$.}
The following proposition formulates the Monte Carlo synthesis using the Monte  Carlo  PSI  operator $Q_f^K$ approximating the Gramian operator  $Q_f$ (Corollary \ref{Ex_Q_E}), and the frame operator $S_f$.

\begin{proposition}
\label{MCsynthesis}
	$V_f^{*K} F = S_f^{-1}V_f^*Q_f^K F.$
\end{proposition}

\begin{proof}
By linearity, it is enough to prove for $K=1$. By the fact that $S_f=V_f^*V_f$, 
  \textcolor{black}{
 \[\begin{split}
    S_f^{-1}V_f^*Q_f^1 F & = S_f^{-1}V_f^* \norm{\psi}_1F(g)K(g,\cdot) \\ & = \norm{\psi}_1F(g) S_f^{-1} V_f^*  V_f(f_g)= \norm{\psi}_1F(g) f_g = V_f^{*1} F. \end{split}\]}
 
\end{proof}

Next, we show that $V_f^{*K}F$ approximates $V_f^*[\psi F]$.

\begin{proposition}[Synthesis Monte Carlo approximation rate]
\label{Delta_approx3}
Let $f$ be a bounded continuous frame with frame bounds $A,B$, and $\norm{f_g}_{\cH}\leq C$. Let $\psi\in L^1(G)$ be an envelope. Then
	\[\mathbb{E}\Big( \norm{V_f^{*K}F-V_f^* [\psi F]}_{\cH}^2 \Big) \leq \frac{\norm{\psi}_1}{K} \frac{B}{A}C^2 \norm{F}_2^2.\]
\end{proposition}

\begin{proof}
 By Proposition \ref{MCsynthesis} and Lemma \ref{Lemm_PI} of Appendix \ref{Pseudo inverse of frame analysis and synthesis operators}
 \begin{equation}
   \begin{split}
\norm{V_f^{*K} F-V_f^*[\psi F]}_{\cH} &= \norm{S_f^{-1}V_f^*Q_f^K F-S_f^{-1}V_f^*Q_f[\psi F]}_{\cH}\\
& = \norm{V_f^+\Big(Q_f^K F-Q_f[\psi F]\Big)}_{\cH} \\
& \leq \norm{V_f^+}\norm{Q_f^K F-Q_f[\psi F]}_2  \leq A^{-1/2}\norm{Q_f^K F-Q_f[\psi F]}_2 .
\end{split}
\label{eq:SynthBound}
 \end{equation}
Indeed, by the frame bound $A^{1/2}\norm{s}_{\cH} \leq \norm{V_f[s]}_2$ for $s=V_f^+F$,
\[\norm{V_f^+F}_{\cH} \leq A^{-1/2}\norm{V_fV_f^+F}_2 =  A^{-1/2}\norm{P_{V_f(\cH)}F}_2 \leq  A^{-1/2}\norm{F}_2.\]
Now, the result follows from  Corollary \ref{Ex_Q_E}. 

\end{proof}

\section{\textcolor{black}{Stochastic phase space signal processing of continuous signals}}
\label{Monte Carlo signal processing in phase space}

In this subsection we formulate and analyze the  Monte Carlo approximations (\ref{eq:PSSPA2MC2AA})--(\ref{eq:PSSPA2MC2BB}) of the signal processing in phase space pipelines (\ref{eq:PSSPS22}) and (\ref{eq:PSSPA2}).  In Subsection \ref{Expected error in stochastic phase space signal processing} we bound the expected value of the error, and in Subsection \ref{Error bounds in high probability} we bound the concentration of measure of the error.

\subsection{Definition of stochastic phase space signal processing}

For Parseval frames (Definition \ref{CSSframe}.6), the frame operator is $S_f=I$, and hence signal processing in phase space takes the form $\mathcal{P}_{f,T,r}:= V_f^* T r\circ V_f$. We call $\mathcal{P}_{f,T,r}$ \emph{Parseval signal processing in phase space} even if $f$ is non-Parseval.
For non-Parseval frames, synthesis and analysis signal processing in phase space involve the multiplications $\mathcal{P}_{f,T,r}S_f^{-1}$ and $S_f^{-1}\mathcal{P}_{f,T,r}$ respectively. Since $\norm{S_f^{-1}}_2\leq A^{-1}<\infty$, it is enough to bound the error entailed by randomly approximating $\mathcal{P}_{f,T,r}$, and then multiply the bound by $\norm{S_f^{-1}}_2$ for non-Parseval pipelines. We hence focus only on Parseval signal processing in our analysis. 

As discussed in Subsection \ref{Sampling in phase space}, sampling in phase space requires enveloping. We hence formulate the following list of signal processing pipelines.
\begin{definition}
\label{SP_pipe}
Let $f$ be a continuous frame over the phase space $G$,  $T$ be a bounded operator in $L^2(G)$,  $r:\CC\rightarrow\CC$, and $\psi,\eta\in L^1(G)$ two envelopes. Let $s\in\cH$ denote a generic signal.
\begin{enumerate}
\item 
A \emph{signal processing pipeline} is defined  by
\begin{equation}
  \label{eq:syn_ana-1}  
  \mathcal{P}_{f,T,r}s = V_f^* T (r\circ V_f[s]).
\end{equation}
    \item 
   An \emph{output enveloped  signal processing pipeline} is defined  by
\begin{equation}
  \label{eq:syn_ana}
  \mathcal{P}^{\psi}_{f,T,r}s = V_f^*\psi T (r\circ V_f[s]).
\end{equation}

    \item
  An  \emph{input-output enveloped signal processing pipeline} is defined by
\begin{equation}
 \label{eq:syn_ana1}
 \mathcal{P}^{\psi; \eta}_{f,T,r}s = V_f^*\psi T \Big(\eta r\circ V_f[s]\Big).
\end{equation}
\end{enumerate}
\end{definition}

The following list of Monte Carlo approximations correspond to the pipelines of Definition \ref{SP_pipe}.

\begin{definition}
\label{def:SSP}
Let $\mathcal{P}_{f,T,r}$ be a signal processing pipeline,  $\psi$ and $\eta$ two envelopes, and $K,L\in\NN$. Let $s\in\cH$ denote a generic signal.
\begin{enumerate}
\item
The \emph{output stochastic signal processing pipeline} is defined by
\begin{equation}
  \label{eq:syn_ana_S}  
  [\mathcal{P}s]_{f,T,r}^{\psi, K} = V_f^{*\psi, K }\big(T (r\circ V_f[s])\big) 
\end{equation}

\item
For a  phase space integral operator $T$, the  \emph{input-output stochastic signal processing pipeline} is defined  by
\begin{equation}
   \label{eq:syn_ana1_S}  
   [\mathcal{P}s]_{f,T,r}^{\psi, K ;\eta,L} = V_f^{*\psi,K}\big(T^{\eta, L}(r\circ V_f[s])\big)
\end{equation}
\end{enumerate}
\end{definition}

We typically fix $f$,  $T$, and $r$, in which case we omit them from the pipeline notation and  denote  $\mathcal{P}$, $\mathcal{P}^{\psi}$, $\mathcal{P}^{\psi,K}$ etc.
Equations (\ref{eq:PSSPA2MC2AA})--(\ref{eq:PSSPA2MC2BB}) give explicit formulas for the synthesis and analysis formulations of  $[\mathcal{P}s]^{\psi, K ;\eta,L}$ and $[\mathcal{P}s]^{\psi, K}$, based on the samples in phase space.
As noted in Subsection \ref{Randomized quadrature approximations of continuous frames}, the pipeline (\ref{eq:syn_ana_S}) is useful for multipliers, shrinkage, and phase vocoder, and the pipeline (\ref{eq:syn_ana1_S}) is useful for PSI operators.

\subsection{Expected error in stochastic phase space signal processing}
\label{Expected error in stochastic phase space signal processing}

In the following, we estimate the error of the stochastic methods.

\begin{theorem}
\label{Phase_op_Monte1}
Let $f$ be a bounded continuous frame with bounds $A,B$, and $\norm{f_g}_{\cH}\leq C$.
Let $T$ be a bounded operator in $L^2(G)$, and $r:\CC\rightarrow\CC$ satisfy $\abs{r(x)}\leq E\abs{x}$ for some $E\geq 0$ and every $x\in\CC$. Let $\psi$ and $\eta$ be two envelopes, and $K,L\in\NN$. Then, for every signal $s\in\cH$, the following two properties hold.
\begin{enumerate}
\item
Output enveloped signal processing stochastic approximation:
\[
\mathbb{E}\Big(\norm{[\mathcal{P}s]^{\psi, K}-[\mathcal{P}s]^{\psi}}_{\cH}^2\Big)
\leq
\frac{\norm{\psi}_1}{K}A^{-1}B^2C^2 E^2\norm{T}_2^2 \norm{s}_{\cH}^2.
\]
\item
Input-output enveloped signal processing stochastic approximation: \newline
if $T$ is a uniformly bounded PSI operator with bound $D$, then
\[
\begin{split}
 & \mathbb{E}\Big(\norm{[\mathcal{P}s]^{\psi, K;\eta,L}-[\mathcal{P}s]^{\psi;\eta}}_{\cH}^2\Big)\\
 & \leq
4\frac{{\norm{\eta}_1}}{L}D^2 B^2 E^2 \norm{s}_{\cH}^2
+16\frac{{\norm{\psi}_1}{\norm{\eta}_1}}{KL}A^{-1}B^2 C^2 D^2E^2\norm{s}_{\cH}^2 \\
 & \quad   +16\frac{\norm{\psi}_1}{K}A^{-1}B^2 C^2 E^2\norm{T}_2^2 \norm{s}_{\cH}^2 \\
	 & = O\left(\frac{\norm{\eta}_1}{L}+\frac{\norm{\psi}_1}{K} + \frac{{\norm{\psi}_1}{\norm{\eta}_1}}{KL}\right)\norm{s}_{\cH}^2. 
\end{split}
\]

\end{enumerate}
\end{theorem}

We use the following simple observation to prove Theorem \ref{Phase_op_Monte1}.

\begin{lemma}
\label{simple_lemma}
Let $Z_1,Z_2,Z_3$ be non-negative real-valued random variables such that
 $Z_1\leq Z_2+Z_3$ 
pointwise in the sample set. Then,
\begin{equation}
    \mathbb{E}(Z_1^2) \leq 4\mathbb{E}(Z_2^2) + 4\mathbb{E}(Z_3^2).
    \label{eq:lem_21}
\end{equation}
\end{lemma}
\begin{proof}
We have  $Z_1\leq 2\max\{Z_2,Z_3\}$, where the maximum is pointwise in the sample space. Therefore, $Z_1^2\leq 4\max\{Z_2^2,Z_3^2\} \leq 4Z_2^2+4Z_3^2$, and (\ref{eq:lem_21}) follows. 
\end{proof}

\begin{proof}[Proof of Theorem \ref{Phase_op_Monte1}]

We prove 2, and note that 1 is simpler and uses similar techniques. Denote by ${\bf g}=\{g^1,\ldots,g^K\}$ the output samples underlying $V_f^{*\psi,K}$, and by ${\bf y}=\{y^1,\ldots,y^L\}$ the input samples underlying $T^{\eta,L}$.
Denote $F=r(V_f[s])\in L^2(G)$.
By the triangle inequality, and by the fact that  $\norm{V_f}=\norm{V_f^*} \leq B^{1/2}$  and $0\leq\psi(g)\leq 1$,
\begin{equation}
  \begin{split}
 & \norm{V_f^{*\psi,K}T^{\eta, L}F-V_f^*\psi T[\eta F]}_{\cH} \\
  & \leq \norm{V_f^{*\psi,K}T^{\eta, L}F-V_f^*\psi T^{\eta,L} F}_{\cH}  + B^{1/2}\norm{ T^{\eta,L} F- T[\eta F]}_{2}.  
\end{split}  
\label{eq:proof0}
\end{equation}
 When calculating the conditional expected value of $\norm{V_f^{*\psi,K}T^{\eta, L}F-V_f^*\psi T[\eta F]}_{\cH}^2 $, with respect to a fixed ${\bf g}$ (denoted here by $\mathbb{E}(\ \cdot\  | {\bf g})$), we use  Lemma \ref{simple_lemma} and Proposition \ref{Input_proposition} to get
\[\begin{split}
 & \mathbb{E}\Big(  \norm{V_f^{*\psi,K}T^{\eta, L}F-V_f^*\psi T[\eta F]}_{\cH}^2  \Big| {\bf g}\Big)
\\
 & \leq
4\frac{{\norm{\eta}_1}}{{L}}D^2 B \norm{F}_2^2
+4 \mathbb{E} \Big( \norm{V_f^{*\psi,K}T^{\eta, L}F-V_f^*\psi T^{\eta,L} F}_{\cH}^2 \Big| {\bf g}\Big).
\end{split}\]
Thus
\[\begin{split}
 & \mathbb{E}\Big(  \norm{V_f^{*\psi,K}T^{\eta, L}F-V_f^*\psi T[\eta F]}_{\cH}^2 \Big) \\
 & \leq
4\frac{{\norm{\eta}_1}}{{L}}D^2 B \norm{F}_2^2
+4 \mathbb{E} \Big( \norm{V_f^{*\psi,K}T^{\eta, L}F-V_f^*\psi T^{\eta,L} F}_{\cH}^2 \Big).
\end{split}\]
 Note that Fubini–Tonelli theorem is satisfied in the computation of \newline
 $\mathbb{E} \Big( \norm{V_f^{*\psi,K}T^{\eta, L}F-V_f^*\psi T^{\eta,L} F}_{\cH}^2 \Big)$ as a repeated integral of ${\bf y}$ and ${\bf g}$, since the integrand is positive and the measure is $\sigma$-finite.

Next, by Proposition \ref{Delta_approx3},
\[\mathbb{E} \Big( \norm{V_f^{*\psi,K}T^{\eta, L}F-V_f^*\psi T^{\eta,L} F}_{\cH}^2 \Big| {\bf y}\Big) \leq  \frac{{\norm{\psi}_1}}{K}A^{-1}BC^2\norm{T^{\eta,L}F}_{2}^2.\]
Now,
\[\norm{T^{\eta,L}F}_{2} \leq   \norm{T^{\eta,L}F- T[\eta F]}_{2} + \norm{T[\eta F]}_{2},\]
so by Lemma \ref{simple_lemma}, by Proposition \ref{Input_proposition}, and by the fact that $0\leq\eta(y)\leq 1$,
\[\begin{split}
 & \mathbb{E} \Big( \norm{V_f^{*\psi,K}T^{\eta, L}F-V_f^*\psi T^{\eta,L} F}_{\cH}^2\Big)\\
 & \leq  \frac{{\norm{\psi}_1}}{{K}}A^{-1}B C^2 \Big(  4\frac{{\norm{\eta}_1}}{{L}}D^2\norm{F}_2^2 + 4\norm{T}_2^2\norm{F}_2^2  \Big).
\end{split}\]

Altogether,
\[
\begin{split}
 & \mathbb{E}\Big(  \norm{V_f^{*\psi,K}T^{\eta, L}F-V_f^*\psi T[\eta F]}_{\cH}^2 \Big)  \\
& \leq 
 4\frac{{\norm{\eta}_1}}{L} D^2 B \norm{F}_2^2
+16\frac{{\norm{\psi}_1}}{K}A^{-1}B C^2  \frac{{\norm{\eta}_1}}{L}D^2\norm{F}_2^2\\
& \quad
+16\frac{{\norm{\psi}_1}}{K}A^{-1}B C^2 \norm{T}_2^2\norm{F}_2^2. 
\end{split}
\]
The claim now follows from $\norm{F}_2^2=\norm{r\circ V_f[s]}_2^2 \leq  E^2 B\norm{s}_{\cH}^2$.

\end{proof}

\subsection{Concentration of error in stochastic phase space signal processing}
\label{Error bounds in high probability}

Propositions \ref{Input_proposition} and \ref{Delta_approx3}, and Theorem \ref{Phase_op_Monte1} estimate the average square error of the stochastic approximations.
In this subsection, we formulate the approximation results as bounds on the error that hold in high probability. We show how to apply the classical concentration of measure estimates, Markov's inequality, and Bernstein's inequality, in our setting.

\subsubsection{A Bernstein inequality in Hilbert spaces}

 In the following version of Bernstein's inequality, we define expected values of weakly integrable random vectors $v$ over the sample set $G$ using the weak integral (Definition \ref{D:week_vec_int}) as $\mathbb{E}^{\rm w}(v)=\int_G^{\rm w}v(g)d\mu(g)$. The following version of Bernstein's inequality is a direct result of \cite[Theorem 2.6]{Berns}, and is proved in Appendix \ref{Hilbert space Berntein inequality}. 
 
 \begin{theorem}[Hilbert space Bernstein's inequality]
\label{H_Bern}
Let $\cH$ be a separable Hilbert space, and $G$ a probability space.
Let $\{v_k\}_{k=1}^K:G^K\rightarrow\cH^K$ be a finite sequence of independent random weakly integrable vectors.
Suppose that for every $k=1,\ldots,K$, $\mathbb{E}^{\rm w}(v_k)=0$ and $\norm{v_k}_{\cH}\leq B$ a.s. and assume that \textcolor{black}{$\rho_K^2> \sum_{k=1}^K \mathbb{E}\norm{v_k}_{\cH}^2$} for some constant $\rho_K\in\RR$. Then, for every $0\leq t \leq \rho_K^2/B$,
\begin{equation}
\label{eq:H_bern}
   P\left( \norm{\sum_{k=1}^K v_k}_{\cH}\geq t\right)\leq \exp\left(-\frac{t^2}{8\rho_K^2}+\frac{1}{4}\right). 
\end{equation}
\end{theorem}

We note that existing variants of Bernstein's inequality in infinite dimensional Hilbert spaces are not adequate for us. For example, the operator Bernstein's inequality of \cite{Op_Bern} is limited to trace class operators, and thus does not even include the identity. 

\subsubsection{Concentration of error results}

Markov type concentration of error results can be derived from Propositions \ref{Input_proposition} and \ref{Delta_approx3}, and Theorem \ref{Phase_op_Monte1}, by multiplying the error bound by $\d^{-1}$, and replacing the expected value with an event that has probability at least $(1-\d)$. For example, the Markov type concentration of error version of Proposition  \ref{Delta_approx3}
reads
\[ \norm{V_f^{*K}F-V_f^* [\psi F]}_{\cH}^2  \leq \frac{\norm{\psi}_1}{K} \frac{B}{A}C^2 \norm{F}_2^2\d^{-1}\]
in probability more that $(1-\d)$.

The following proposition summarizes the Markov and Bernstein types concentration of error bounds in output stochastic  signal processing.

\begin{theorem}[Output signal processing  concentration of error]
\label{Delta_approx300}
Let $f$ be a bounded continuous frame with frame bounds $A$ and $B$, and with $\norm{f_g}_{\cH}\leq C$, let $\psi$ be an envelope, and $K\in\NN$. Let $T$ be a bounded  operator in $L^2(G)$, and $r:\CC\rightarrow\CC$ satisfy $\abs{r(x)}\leq E \abs{r(x)}$ for every $x\in\CC$, where $E>0$. Let $s\in \cH$ and $0<\delta<1$. 
 Then,
	with probability more than $1-\delta$, we have
	\begin{equation}
	    \label{eq:BBV1}
	     \norm{[\mathcal{P}s]^{\psi, K}-[\mathcal{P}s]^{\psi}}_{\cH} \leq  \frac{\sqrt{\norm{\psi}_1}}{\sqrt{K}} A^{-1/2}BCE\norm{T}_2\norm{s}_{\cH} \kappa(\delta),
	\end{equation}
where $\kappa(\d)$ can be chosen as one of the following two options.
\begin{enumerate}
	\item 
	\emph{Markov type error bound:} $\kappa(\d)=\delta^{-\frac{1}{2}}$.
	\item
	\emph{Bernstein type error bound:} $\kappa(\d)=2\sqrt{2}\sqrt{\ln\Big(\frac{1}{\d}\Big) +\frac{1}{4}}$ in case $\norm{T}_{\infty}<\infty$ and $K$ satisfies
	\begin{equation}
	   K \geq  \norm{\psi}_1\Big(\frac{C}{B^{1/2}} \frac{\norm{T}_{\infty}}{\norm{T}_2} + \frac{B^{1/2}}{C\norm{\psi}_1} \Big)^2 \k(\d)^2.
	   \label{eq:Klb}
	\end{equation}
\end{enumerate}
\end{theorem}


\begin{proof}
We prove 2 and note that 1 is simpler and based on Markov's inequality.
Denote $F=T r(V_f[s])$.
Below, we use the following bounds

\begin{equation}
    \label{eq:TV_bound0}
    \norm{F}_{\infty} =\norm{T r(V_f[s])}_{\infty} \leq \norm{T}_{\infty} EC\norm{s}_{\cH},
\end{equation}
\begin{equation}
    \label{eq:TV_bound1}
  \norm{F}_2=\norm{T r(V_f[s])}_{2} \leq \norm{T}_2 E B^{1/2}\norm{s}_{\cH} .
\end{equation}
and
\begin{equation}
\label{eq:Eless}
 B^{1/2}C \norm{F}_{\infty}+ \frac{1}{\norm{\psi}_1}B\norm{F}_2\leq J,
\end{equation}
where
\[J= B^{1/2}C^2E \norm{T}_{\infty} \norm{s}_{\cH}+ \frac{1}{\norm{\psi}_1}B^{1.5}E \norm{T}_2\norm{s}_{\cH}.\]

We use Theorem \ref{H_Bern} as follows.
Define the independent random vectors
\[v_k:G^k\rightarrow L^2(G), \quad v_k(\mathbf{g})= \frac{1}{K} \big(Q_f^1(g^k)  F -  Q_f \psi F \big), \quad k=1,\ldots, K\]
where the sample set is $\{G^k\ ;\ \prod_{k=1}^K \frac{\psi(g^k)}{\norm{\psi}_1}dg^k\}$. 
By Corollary \ref{Ex_Q_E}, $\mathbb{E}^{\rm w}(v_k)=\mathbb{E}(v_k)=0$, and 
$\mathbb{E}(\norm{v_k}_2^2) \leq \frac{\norm{\psi}_1}{K^2}BC^2\norm{F}^2_2$. Therefore, by (\ref{eq:TV_bound1}),
\[\sum_{k=1}^K \mathbb{E}(\norm{v_k}_2^2) \leq \frac{\norm{\psi}_1}{K}BC^2\norm{F}^2_2 \leq \frac{\norm{\psi}_1}{K}B^2C^2E^2 \norm{T}^2_2\norm{s}^2_{\cH} .\]
Moreover, by Proposition \ref{Gram_Ker}, Example \ref{Ex:Qf_bound}, and (\ref{eq:Eless}), for every $g^k\in G$
\[
\begin{split}
  \norm{v_k}_2 & \leq \frac{1}{K}\Big(\norm{ \norm{\psi}_1K(g^k,\cdot)F(g^k) }_2 + \norm{Q_f \psi F}_2\Big) \\
  & \leq \frac{\norm{\psi}_1}{K}\Big(  B^{1/2}C \norm{F}_{\infty}+ \frac{1}{\norm{\psi}_1}B\norm{F}_2\Big) \leq \frac{\norm{\psi}_1}{K} J. 
\end{split}
\]
Hence, by Theorem \ref{H_Bern},
for every
$0 \leq t\leq B^2C^2E^2 \norm{T}^2_2\norm{s}^2_{\cH}/J$
\begin{equation}
P\Big( \norm{Q_f^K F - Q_f(\psi F)}_2 \geq t \Big) \leq \exp\left(-\frac{t^2}{8B^2C^2E^2 \norm{T}^2_2\norm{s}^2_{\cH}}\frac{K}{\norm{\psi}_1} + \frac{1}{4}\right).
\label{eq:BBV2}
\end{equation}
Now, set
\[\d=\exp\left(-\frac{t^2}{8B^2C^2E^2 \norm{T}^2_2\norm{s}^2_{\cH}}\frac{K}{\norm{\psi}_1} + \frac{1}{4}\right),\]
or equivalently
\[
t=\sqrt{8}\sqrt{-\ln(\delta)+\frac{1}{4}}BCE\norm{T}_2\norm{s}_{\cH}\frac{\sqrt{\norm{\psi}_1}}{\sqrt{K}},
\]
and demand $0 \leq t\leq B^2C^2E^2 \norm{T}^2_2\norm{s}^2_{\cH}/J$, namely,
\[\sqrt{8}\sqrt{-\ln(\delta)+\frac{1}{4}}BCE\norm{T}_2\norm{s}_{\cH}\frac{\sqrt{\norm{\psi}_1}}{\sqrt{K}} \leq B^2C^2E^2 \norm{T}^2_2\norm{s}^2_{\cH}/J.\]
%
%
%
%
%
This gives, in probability at least $(1-\delta)$,
\[\norm{Q_f^K F - Q_f\psi F}_{\cH} \leq BCE\norm{T}_2\norm{s}_{\cH}\frac{\sqrt{\norm{\psi}_1}}{\sqrt{K}}\kappa(\delta),\]
whenever $k$ satisfies (\ref{eq:Klb}).
Last, using Proposition \ref{MCsynthesis} and (\ref{eq:SynthBound}), we get
\[
\begin{split}
 \norm{[\mathcal{P}s]^{\psi, K}-[\mathcal{P}s]^{\psi}}_{\cH} & =
\norm{S_f^{-1}V_f^*Q_f^K F - S_f^{-1}V_f^*Q_f\psi F}_{\cH}\\
& \leq A^{-1/2}\norm{Q_f^K F - Q_f\psi F}_{\cH}
\end{split}
\]

%
%
%
%
\end{proof}

\section{Stochastic phase space signal processing of discrete signals}
\label{Discrete stochastic phase space signal processing}

\textcolor{black}{In previous sections we showed how to randomly discretize phase space. 
In   Theorems \ref{Phase_op_Monte1} and \ref{Delta_approx300}, when the number of samples satisfy $K,L=Z \max\{\norm{\psi}_1,\norm{\eta}_1\}$, for  $Z>0$, the approximation errors are of order $O(Z^{-1/2})$.
In this section, we additionally discretize the signal space $\cH$ to a finite dimensional subspace $V_M\subset \cH$ of dimension/resolution $M\in\NN$.} 
The main goal is to relate the choices of $\norm{\psi}_1$ and $\norm{\eta}_1$ to the \textcolor{black}{resolution} $M$. We introduce a class of frames, called \emph{linear volume discretizable (LVD) frames}, for which there are envelopes $\psi_M$ and $\eta_M$ with $\norm{\psi_M}_1,\norm{\eta_M}_1=O(M)$ that contain most of the energy of $V_f[s_M]$ for every $s_M\in V_M$. Thus, a stochastic signal processing  method for LVD frames requires $K,L=ZM$ samples, with $Z>0$, for the approximation error to be  $O(\frac{1}{\sqrt{Z}})$.

\subsection{Discrete signals and linear volume discretization of frames}

We treat discrete signals as embedded in the Hilbert space of signals $\cH$. A discrete signal is an element of a finite dimensional subspace of $\cH$. On the one hand, we can analyze discrete signals directly in $\cH$. On the other hand, discrete signals are determined by a finite number of scalars, so they are well adapted to numerical analysis. In our analysis, we sometimes restrict ourselves to a class of signals $\mathcal{R}\subset \cH$ which need not be a linear space. We typically consider $\mathcal{R}$ defined by imposing a restriction on signals in $\cH$ which is natural for real life signals of some type. 
\textcolor{black}{
\begin{definition}
Let $\cH$ be a Hilbert space that we call the \emph{signal space}.
A \emph{class of signals} $\mathcal{R}\subset\cH$ is a (possibly non-linear) subset of $\cH$. A \emph{sequence of discretizations} of $\mathcal{R}$ is a sequence of (generally non-linear) subspaces $\{V_M\subset \mathcal{H}\}_{M=1}^{\infty}$ that satisfies the following condition: for every $s\in \mathcal{R}$ there is a sequence $\{s_M\in V_M\}_{M=1}^{\infty}$ such that
\[\lim_{M\rightarrow\infty}\norm{s_M-s}_{\cH}=0.\]
 The \emph{resolution} ${\rm dim}(V_M)$ of $V_M$ is defined to be the dimension of ${\rm span}V_M$.
\end{definition}
}




The idea in discretizing \textcolor{black}{a continuous frame} is to find an envelope $\psi_M$ for each discrete space $V_M$ such that for any $s_M\in V_M$, the approximation error of $V_f[s_M]$ by $\psi_M V_f[s_M]$ is controlled.
 The envelopes $\psi_M$ are interpreted as  covering domains $G_M\subset G$ in which most of the energy of functions from $V_f [V_M]$ resides. 



\begin{definition}
\label{D:linear area discretizable}
Let $f:G\rightarrow\cH$ be a continuous frame.
Let $\mathcal{R}\subset \cH$ be a class of signals, and $\{V_M\}_{m=1}^{\infty}$ a discretization of $\mathcal{R}$. 
\begin{enumerate}
	\item 
	The continuous frame $f$ is called \emph{linear \textcolor{black}{volume} discretizable (LVD)} with respect to the class $\mathcal{R}$ and the discretization $\{V_M\}_{M=1}^{\infty}$, if for every error tolerance $\e>0$ there is a constant $C_{\epsilon}>0$ and $M_0\in\NN$, such that for any $M\geq M_0$ there is an envelope $\psi_M$ with 
\begin{equation}
\norm{\psi_M}_1 \leq C_{\epsilon}{\rm dim}(V_M)
\label{eq:lin_area1}
\end{equation}
such that for any $s_M\in  V_M$,
\begin{equation}
\frac{\norm{V_f[s_M] - \psi_M V_f[s_M]}_2}{\norm{V_f[s_M]}_2} < \e.
\label{eq:lin_area2}
\end{equation}
\item
For a linear \textcolor{black}{volume} discretizable continuous frame $f$ with respect to $\mathcal{R}$ and $\{V_M\}_{M=1}^{\infty}$, and a fixed tolerance $\e>0$ with a corresponding fixed $C_{\epsilon}$ and envelope sequence $\{\psi_M\}_{M=1}^{\infty}$ satisfying (\ref{eq:lin_area1}) and (\ref{eq:lin_area2}), we call $f$ together with $\mathcal{R}$, $\{V_M\}_{m=1}^{\infty}$,  and $\{\psi_M\}_{M=1}^{\infty}$, an \emph{$\e$-linear volume discretization 
($\e$-LVD)} of $f$.
\end{enumerate}
\end{definition}



\subsection{\textcolor{black}{Error in discrete stochastic phase space signal processing}}
\label{Error in discrete stochastic phase space signal processing}

Next, we study the error in discrete stochastic phase space signal processing of  LVD frames. Since the energy of $V_f[s_M]$ may be shifted after applying an operator $T$ on $V_f[s_M]$, we first introduce the following definition.

\begin{definition}
\label{def:mapE}
Let $G$ be a phase space,  $T$ a bounded linear operator in $L^2(G)$,  $\psi$ and $\eta$ two envelopes, and $\e>0$. We say that $T$ \emph{maps  the energy of $\eta$ to $\psi$ up to $\epsilon$}, if
\begin{equation}
\norm{T\eta - \psi T \eta}_{2\rightarrow 2} \leq \e.
\label{eq:Tmap0b}
\end{equation}
\end{definition}

The next theorem summarizes the expected approximation error in stochastic signal processing with $\e$-LVD frames.

\begin{theorem}
\label{Main_Theorem12}
Let $f$ be a bounded continuous frame with bound $\norm{f_g}_{\cH}\leq C$. Let $r:\CC\rightarrow\CC$ satisfy $\abs{r(x)}\leq E\abs{x} $, where $E>0$.
Suppose that $f$ together with the signal class $\mathcal{R}$, the discretization $\{V_M\}_{m=1}^{\infty}$,  and the envelopes $\{\eta_M\}_{M=1}^{\infty}$, is an $\e$-LVD of $f$, with constant $C_{\epsilon}$. Let $\{\psi_M\}_{M=1}^{\infty}$ be a sequence of envelopes satisfying
\begin{equation}
\norm{\psi_M}_1 \leq C_{\epsilon}{\rm dim}(V_M).
\label{eq:lin_area12b}
\end{equation}
Let $T$ be a bounded operator  on $L^2(G)$ that maps the energy of $\eta_M$ to  $\psi_M$ up to $\epsilon$.
Then, the following two bounds are satisfied for every $s_M\in V_M$.

\begin{enumerate}
	\item $ $
	
	\vspace{-8mm}
	\begin{equation}
	\begin{split}
    & \frac{\mathbb{E}\Big(\norm{ [\mathcal{P}s_M]_{f,T,r}^{\psi_M, K}  - [\mathcal{P}s_M]_{f,T,r} }_{\cH}^2\Big)}{{\norm{s_M}_{\cH}^2} } \\
    & \leq
    4\frac{C_{\epsilon}{\rm dim}(V_M)}{K}A^{-1}B^2C^2 E^2\norm{T}_2^2
    +4 B^2  E^2 (1+2\norm{T}_2)^2 \e^2
\end{split}
	\label{eq:Discrete_SP1}
	\end{equation}
	\item If $T$ is a uniformly square integrable PSI operator with bound $D$, then
\begin{equation}
	\begin{split}
    & \frac{\mathbb{E}\Big(\norm{ [\mathcal{P}s_M]_{f,T,r}^{\psi_M, K;\eta_M,L}  - [\mathcal{P}s_M]_{f,T,r} }_{\cH}^2\Big)}{{\norm{s_M}_{\cH}^2} }  \\
     &   \leq
16\frac{C_{\epsilon}{\rm dim}(V_M)}{L}D^2 B^2 E^2 
+64\frac{C_{\epsilon}^{2}{\rm dim}(V_M)^2}{KL}A^{-1}B^2 C^2 D^2E^2 \\
 & \quad   +64\frac{C_{\epsilon}{\rm dim}(V_M)}{K}A^{-1}B^2 C^2 E^2\norm{T}_2^2  + 4B^2 E^2 (1+\norm{T}_2^2)\e^2.
\end{split}  
	\label{eq:Discrete_SP10}
	\end{equation}
	\end{enumerate}
\end{theorem}

\begin{proof}

We first prove (\ref{eq:Discrete_SP10}).
By Lemma \ref{simple_lemma},
\begin{equation}
  \begin{split}
    & \mathbb{E}\Big(\norm{ [\mathcal{P}s_M]^{\psi_M, K;\eta_M,L} - [\mathcal{P}s_M]_{f,T,r} }_{\cH}^2\Big) \\
     & \leq 4 \mathbb{E}\Big(\norm{[\mathcal{P}s_M]^{\psi_M, K;\eta_M,L}-[\mathcal{P}s_M]^{\psi_M;\eta_M}}_{\cH}^2\Big)  
     + 4\norm{[\mathcal{P}s_M]^{\psi_M;\eta_M} - [\mathcal{P}s_M]}_{\cH}^2.
\end{split}  
\label{eq:T29_1}
\end{equation}
Next, we bound the second term of (\ref{eq:T29_1}). 
\begin{equation}
 \begin{split}
  \norm{[\mathcal{P}s_M]^{\psi_M;\eta_M} - [\mathcal{P}s_M]}_{\cH} 
 & = \norm{V_f^*\psi_M T \Big(\eta_M r\circ V_f[s_M]\Big) - V_f^* T \Big( r\circ V_f[s_M]\Big)}_{\cH} \\
  & \leq B^{1/2}\norm{\psi_M T \Big(\eta_M r\circ V_f[s_M]\Big) -  T \Big( r\circ V_f[s_M]\Big)}_{\cH} \\
   & \leq B^{1/2}\norm{ T \eta_M r\circ V_f[s_M] -\psi_M T \eta_M r\circ V_f[s_M]}_2\\
   & \quad + B^{1/2}\norm{ Tr\circ V_f[s_M] - T \eta_M r\circ V_f[s_M]}_2  \\
 &  \leq   B^{1/2}\norm{ T \eta_M  -\psi_M T \eta_M}_2 \norm{ r\circ V_f[s_M]}_2 \\
 & \quad + B^{1/2}\norm{T}_2\norm{ r\circ V_f[s_M] -  \eta_M r\circ V_f[s_M]}_2\\
 & \leq  B\e E\norm{s_M}_{\cH} + B^{1/2}\norm{T}_2\norm{ (1-\eta_M)r\circ V_f[s_M]}_2 .
\end{split}   
\label{eq:Discrete_SPSO2}
\end{equation}
For the second term of the last line of (\ref{eq:Discrete_SPSO2}), by the  LVD property,
\begin{equation}
    \label{eq:T29_2}
    \norm{ (1-\eta_M)r\circ V_f[s_M]}_2  
\leq E \norm{\eta_M V_f[s_M]-V_f[s_M]}_2 \leq \e E B^{1/2}\norm{s_M}_{\cH}.
\end{equation}
To conclude, (\ref{eq:T29_1}) together with (\ref{eq:Discrete_SPSO2}), (\ref{eq:T29_2}), and Theorem \ref{Phase_op_Monte1}, give (\ref{eq:Discrete_SP10}).

Next, we prove  (\ref{eq:Discrete_SP1}).
As before,
\begin{equation}
  \begin{split}
    & \mathbb{E}\Big(\norm{ [\mathcal{P}s_M]_{f,T,r}^{\psi_M, K}  - [\mathcal{P}s_M]_{f,T,r} }_{\cH}^2\Big) \\
     & \leq 4 \mathbb{E}\Big(\norm{[\mathcal{P}s_M]^{\psi_M, K}-[\mathcal{P}s_M]^{\psi_M}}_{\cH}^2\Big)  
     + 4\norm{[\mathcal{P}s_M]^{\psi_M} - [\mathcal{P}s_M]}_{\cH}^2.
\end{split}  
\label{eq:T29_5}
\end{equation}
We bound the second term of (\ref{eq:T29_5}) using (\ref{eq:Tmap0b}) and  (\ref{eq:T29_2}) by
\begin{equation}
\begin{split}
  &  \norm{[\mathcal{P}s_M]^{\psi_M} - [\mathcal{P}s_M]}_{\cH}
= \norm{V_f^*\psi_M T r\circ V_f[s_M] -V_f^* T r\circ V_f[s_M]
}_{\cH} \\
& \leq B^{1/2} \norm{ T r\circ V_f[s_M] -  T \eta_M r\circ V_f[s_M]}_{\cH}\\
&\quad +B^{1/2} \norm{T \eta_M r\circ V_f[s_M]-\psi_M T \eta_M  r\circ V_f[s_M] }_{\cH}\\
&\quad +B^{1/2} \norm{\psi_M T \eta_M  r\circ V_f[s_M]-\psi_M T  r\circ V_f[s_M] }_{\cH}\\
& \leq B^{1/2} \norm{ T}_2\norm{ r\circ V_f[s_M] -   \eta_M r\circ V_f[s_M]}_{\cH}\\
&\quad  +B^{1/2} \norm{T \eta_M -\psi_M T \eta_M  }_{\cH}\norm{r\circ V_f[s_M]}_{\cH} \\
&\quad  +B^{1/2} \norm{T}_2\norm{\eta_M  r\circ V_f[s_M]-  r\circ V_f[s_M] }_{\cH}\\
& \leq B^{1/2} \norm{T}_2 \e E B^{1/2}\norm{s_M}_{\cH}
+B^{1/2} \e E B^{1/2}\norm{s_M}_{\cH}
+B^{1/2} \norm{T}_2 \e E B^{1/2}\norm{s_M}_{\cH} \\
& =  BE(2\norm{T}_2+1)\norm{s_M}_{\cH}.
\end{split}
    \label{eq:T29_9}
\end{equation}
This, together with (\ref{eq:T29_5}) and Theorem \ref{Phase_op_Monte1}, leads to  (\ref{eq:Discrete_SP1}).
%
%
%
%
%
%
%
%
%
%
%
%

\end{proof}

Next, we formulate concentration of error results for LVD frames. A Markov type concentration of error result can be derived directly from Theorem \ref{Main_Theorem12}. For a Bernstein type error bound, we offer the following theorem only for the output stochastic signal processing pipeline (Definition \ref{def:SSP}.1).

\begin{theorem}
\label{th:ful_lMC_SP}
Consider the setting of Theorem \ref{Main_Theorem12}.1, and suppose that $\norm{T}_{\infty}<\infty$. Let $\delta>0$, 
$\kappa(\d)=2\sqrt{2}\sqrt{\ln\Big(\frac{1}{\d}\Big) +\frac{1}{4}}$, and $K$ satisfy (\ref{eq:Klb}).
Then, in probability more than $(1-\delta)$,
\begin{equation}
  \begin{split}
    & \frac{\norm{ [\mathcal{P}s_M]_{f,T,r}^{\psi_M, K}  - [\mathcal{P}s_M]_{f,T,r} }_{\cH}}{\norm{s_M}_{\cH}} \\
     & \leq  \frac{\sqrt{C_{\epsilon}{\rm dim}(V_M)}}{\sqrt{K}} A^{-1/2}BCE\norm{T}_2 \kappa(\delta)  + BE(2\norm{T}_2+1) \epsilon.
\end{split}  
\label{eq:T30_0}
\end{equation}
\end{theorem}

\begin{proof}
The proof follows from (\ref{eq:T29_9}) and Theorem \ref{Delta_approx300}.2, similarly to the proof of Theorem \ref{Main_Theorem12}.1. 

\end{proof}

\subsection{Discrete stochastic time-frequency signal processing}
\label{Discretization of phase space in time frequency analysis}

In this subsection we present a discretization under which the STFT 
 is LVD. In the companion paper \cite{Ours_app} we present a discretization under which the CWT is linear volume discretizable.  We analyze time signals $s:\RR\rightarrow\CC$ by decomposing them to compact time interval sections. Without loss of generality, we suppose that each signal segment is supported in $[-1/2,1/2]$.  
%
%
%
%
Focusing on one segment, we take the signal class $\mathcal{R}$  as $L^2[-1/2,1/2]$.
%
%
 Let $V_M$ be the space of trigonometric polynomials of order $M$ (namely, finite Fourier series expansions).
In the frequency domain, signals $q\in V_M$ are represented by
\[\hq(z)=\sum_{n=-M}^M c_n {\rm sinc}(z-n)\]
where $c_n$ are the Fourier coefficients of $q$, and ${\rm sinc}$ is the Fourier transform of the indicator function of  $[-1/2,1/2]$.
Consider a window function $f$ supported at the time interval $[-S,S]$ that satisfies the following. There exist constants $C',Y>0$, and $\k>1/2$, such that for every $z>Y$ or $z<-Y$
\begin{equation}
\hf(z) \leq C'\abs{z}^{-\k}.
\label{eq:fSTFT_decay}
\end{equation}

Let $W>0$.
For each $M\in\NN$, we consider the following phase space domain $G_M\subset G$, where $G$ is the STFT time-frequency plane,
\begin{equation}
G_M = \big\{(x,\w)\ \big|\ -WM < \w <WM ,\ \abs{x}< 1/2+ S\big\}.
\label{eq:GM_STFT}
\end{equation}
The area of $G_M$ in the time-frequency plane is
\begin{equation}
\mu(G_M) = 2WM(1+2S).
\label{eq:VolGM_CWT2}
\end{equation}
Denote 
\begin{equation}
\psi_M(g) = \left\{\begin{array}{ccc}
	1 & , & g\in G_M \\
	0 & ,  & g\notin G_M.
\end{array}\right.
\label{eq:psi_STFT}
\end{equation}

\begin{theorem}
\label{lin_vol_STFT}
Under the above setting,
the STFT is LVD with respect to the class $L^2[-1/2,1/2]$ and the discretization $\{V_M\}_{M\in\NN}$, with the envelopes $\psi_M$ defined by (\ref{eq:psi_STFT}) for large enough $W$ that depends only on $\e$ of Definition \ref{D:linear area discretizable}. 
\end{theorem}

\begin{proof}

Let $W>1$.
A direct calculation of the STFT shows
\begin{equation}
\int_{\RR}\abs{V_f[q](\w,x)}^2 dx = \int_{\RR} \abs{\hq(z)}^2\abs{\hf(z-\w)}^2  dz.
\label{eq:Vfs_slice_STFT}
\end{equation}
We consider $\w>0$ and $z>0$, and note that the other cases are similar.
For each value of $\w>MW$, we decompose the integral (\ref{eq:Vfs_slice_STFT}) along $z$ into the two integrals in $z\in(0,(M+\w)/2)$ and $z\in((M+\w)/2,\infty)$.
For $z\in(0,(M+\w)/2)$, 
since $\w\geq MW$ and $z\leq(M+\w)/2$, 
\[z-\w \leq M-\w/2 \leq -M(W/2-1)<0,\]
so $\abs{z-\w}^{-2\k}$ obtains its maximum at $z=(M+\w)/2$.
Thus, by (\ref{eq:fSTFT_decay}),
\[\int_{0}^{(M+\w)/2} \abs{\hq(z)}^2\abs{\hf(z-\w)}^2  dz \leq \norm{q}_2^2  \max_{0 \leq z \leq (M+\w)/2}C^{\prime 2}\abs{z-\w}^{-2\k}\]
\begin{equation}
=\norm{q}_2^2 C^{\prime 2}\abs{(M+\w)/2-\w}^{-2\k}= \norm{q}_2^2 C^{\prime 2}\abs{(M-\w)/2}^{-2\k}
\label{eq:STFTvolume20}
\end{equation}
Integrating the bound (\ref{eq:STFTvolume20}) for $\w\in(WM,\infty)$ gives
\begin{equation}
\begin{split}
\int_{WM}^{\infty}\int_{0}^{(M+\w)/2} \abs{\hq(z)}^2\abs{\hf(z-\w)}^2  dz d\w &= (W-1)^{-2k+1}M^{-2k+1}\norm{q}_2^2 O(1) \\
 & = o_W(1)o_M(1) \norm{q}_2^2.
\end{split}
\label{eq:STFTvolume30}
\end{equation}
For $z\in((M+\w)/2,\infty)$, $\hq$ decays like $M^{1/2}(z-M)^{-1}$. 
Indeed, since $z>M$ 
\begin{equation}
\begin{split}
\sum_{n=-M}^M c_n {\rm sinc}(z-n)  & \leq \norm{\{c_n\}}_2 \sqrt{\sum_{n=-M}^M \frac{1}{(z-n)^2} }
\leq \norm{q}_2 \sqrt{\sum_{n=-M}^M \frac{1}{(z-M)^2} } \\
& 
\leq   2\norm{q}_2\sqrt{M} (z-M)^{-1}.
\end{split}
\label{eq:STFTvolume1}
\end{equation}
Now, by (\ref{eq:Vfs_slice_STFT}) and (\ref{eq:STFTvolume1}),
\begin{equation}
\begin{split}
 \int_{(M+\w)/2}^{\infty} \abs{\hq(z)}^2\abs{\hf(z-\w)}^2  dz & \leq 2\norm{f}_2^2\norm{q}_2^2  \max_{(M+\w)/2 \leq z < \infty} M(z-M)^{-2}\\
& =\norm{f}_2^2\norm{q}_2^2  M \big((\w-M)/2\big)^{-2}.
\end{split}
\label{eq:STFTvolume2}
\end{equation}
Integrating the bound (\ref{eq:STFTvolume2}) for $\w\in(WM,\infty)$ gives
\begin{equation}
\int_{WM}^{\infty}\int_{(M+\w)/2}^{\infty} \abs{\hq(z)}^2\abs{\hf(z-\w)}^2  dz d\w = (W-1)^{-1}\norm{q}_2^2 O(1) .
\label{eq:STFTvolume3}
\end{equation}
Last, the bounds (\ref{eq:STFTvolume30}) and (\ref{eq:STFTvolume3}) are combined to give
$\norm{(I-\psi_M)V_f[q]}_2= o_W(1)\norm{q}_2$, 
so by the frame inequality
\[\frac{\norm{(I-\psi_M)V_f[q]}_2}{\norm{V_f[q]}_2}= o_W(1).\]
This means that given $\e>0$, we may choose $W$ large enough to guarantee $\frac{\norm{(I-\psi_M)V_f[q]}_2}{\norm{V_f[q]}_2} < \e$, 
 and also guarantee that for every $M\in\NN$, 
$\norm{\psi_M}_1\leq C_{\epsilon} M$, 
with $C_{\epsilon}= 2W(1+2S)$ by (\ref{eq:VolGM_CWT2}).
%

\end{proof}



\section{\textcolor{black}{Applications of stochastic signal processing of continuous signals}}

In this section, we introduce two  applications of the theory developed in this paper: integration of continuous frames and stochastic phase space diffeomorphism.

\subsection{Integration of linear volume discretizable frames}
\label{Integration of coherent state systems}

Here, we show how to integrate a set of LVD continuous frames into one continuous LVD frame, while retaining all stochastic approximation bounds of a single LVD frame. We first show how to integrate frames.


\begin{proposition}
\label{prop:int1}
Let $G$ and $U$ be two topological spaces with $\sigma$-finite Borel measures $\mu_G$ and $\mu_U$ respectively, with $\mu_U(U)=1$.
Let $A,B,C>0$ and $\cH$ be a Hilbert space.
For each $u\in U$, let $f_{\cdot,u}:g\mapsto f_{g,u}$ be a bounded continuous frame over the phase space $G$ and the signal space $\cH$, with frame constants $A,B$ and bound $\norm{f_{g,u}}_{\cH}\leq C$.    Suppose  that the mapping $f:(g,u)\mapsto f_{g,u}$ is continuous.
Then $f$ is a bounded continuous frame over the phase space $G\times U$, with frame constants $A,B$ and bound $\norm{f_{g,u}}_{\cH}\leq C$.
\end{proposition}

\begin{proof}
 Consider the mapping $V_f$ that maps $s\in\cH$ to the function 
\[V_f[s]:(g,u)\mapsto \ip{s}{f_{g,u}}.\]
By continuity of $(g,u)\mapsto f_{g,u}$, $V_f[s]$ is continuous for every $s\in \cH$. Indeed
\[\abs{V_f[s](g,u) - V_f[s](g',u')} = \abs{\ip{s}{f_{g,u}} - \ip{s}{f_{g',u'}}} \leq \norm{s}_{\cH}\norm{f_{g,u} - f_{g',u'}}_{\cH}.\]
Thus, for every $s\in \cH$, $V_f[s]:G\times U \rightarrow \CC$ is a measurable function. 
For each $u\in U$, denote by $V_{f_{\cdot,u}}$ the analysis operator corresponding to the continuous frame $f_{\cdot,u}$.
By Fubini-Tonelli theorem, for every signal $s\in\cH$ 
\[\norm{V_f[s]}_2^2=\iint_{G\times U} \abs{\ip{s}{f_{g,u}}}^2 d(g,u)=\int_U\int_G \abs{\ip{s}{f_{g,u}}}^2 dgdu = \int_U \norm{V_{f_{\cdot,u}}[s]}_2^2 du.\]
Therefore,
\[ A\norm{s}_{\cH}^2 =\int_U A \norm{s}_{\cH}^2 du  \leq \norm{V_f[s]}_2^2 \leq \int_U B \norm{s}_{\cH}^2 du = B \norm{s}_{\cH}^2,\]
and $\{f_{g,u}\}_{(g,u)\in G\times U}$ is a continuous frame with frame bounds $A,B$.

\end{proof}

Next, we show that the LVD property is retained under integration of frames.

\begin{proposition}
\label{th:int_LVD}
Consider the setting of Proposition \ref{prop:int1}.
Let $\mathcal{R}\subset\cH$ be a signal class, and $V_M$ a discretization of $\mathcal{R}$. Suppose that for every $u\in U$, $f_{\cdot,u}$ is an LVD frame. Let $\e>0$ and $\{\psi_M\in L^1(G)\}_{M=1}^{\infty}$ a sequence of envelopes.  Suppose that for every $u\in U$, $f_{\cdot,u}$ is $\e$-LVD with respect to the envelopes $\{\psi_M\in L^1(G)\}$ and the constant $C_{\e}$. For each $M$, denote by $\psi_M\in L^1(G\times U)$ the envelope $(g,u)\mapsto \psi_M(g)$. 
Then,  $f$ is an $\e$-LVD frame with respect to the envelopes $\{\psi_M\in L^1(G\times U)\}$ and the bound $C_{\e}$.
\end{proposition}

\begin{proof}
By Fubini-Tonelli theorem
\begin{equation}
\norm{V_{f}[s_M]-\psi_M V_{f}[s_M]}_2^2  = \int_U \norm{V_{f_{\cdot,u}}[s_M]-\psi_M V_{f_{\cdot,u}}[s_M]}_2^2 du  \leq \e^2 \norm{V_{f}[s]}_2^2. 
\label{eq:int_frame1}
\end{equation}

\end{proof}

Last, we show how to integrate operators in phase space, and show that mapping the energy between envelopes up to $\e$ (Definition \ref{def:mapE}) is preserved under integration.

\begin{proposition}
\label{prop:T_int}
Consider the setting of Proposition \ref{prop:int1}.
Let $\mathcal{R}\subset\cH$ be a signal class, and $V_M$ a discretization of $\mathcal{R}$. Suppose that for every $u\in U$, $f_{\cdot,u}$ is an  LVD frame. Let $\e>0$ and $\{\eta_M\in L^1(G)\}_{M=1}^{\infty}$ and $\{\psi_M\in L^1(G)\}$ two sequences of envelopes.  Suppose that for every $u\in U$, $f_{\cdot,u}$ is $\e$-LVD with respect to the envelopes $\{\eta_M\in L^1(G)\}$ and the constant $C_{\e}$.
For each $u\in U$, let $T_u$ be a bounded operator in $L^2(G)$ with $\norm{T_u}_{L^2(G)}\leq C_T$.  Suppose that for every $M\in\NN$ and a.e. $u\in U$,  $T_u$ maps the energy of $\eta_M$ to $\psi_M$ up to $\epsilon$.
Let  $T$ be the operator in $L^2(G\times U)$ defined for $F\in L^2(G\times U)$ by
\[TF(g,u) = T_u F(\cdot,u)(g).\]
Then $T$ is bounded with $\norm{T}_{L^2(G\times U)}\leq C_T$, and maps  the energy of $\eta_M\in L^1(G\times U)$ to $\psi_M\in L^1(G\times U)$ up to $\epsilon$.
\end{proposition}


Under the assumptions of Proposition \ref{prop:T_int}, $f$ and $T$ satisfy the conditions of Theorems \ref{Main_Theorem12} and \ref{th:ful_lMC_SP}. This means that the number of random samples in the stochastic method in $f$, required for a given accuracy, is comparable to the number of samples required for each $\{f_{\cdot,u}\}_{u\in U}$. Namely,  the addition of the new feature direction $U$ to the phase space $G$ does not entail an increase in computational complexity, and the approximation of the continuous method by the discrete Monte Carlo method is of order $O(\frac{\sqrt{{\rm dim}(V_M)}}{\sqrt{K}})$.

\textcolor{black}{
The above procedure of integrating continuous frames can be carried out when the definition of a certain continuous frame depends on some free parameters $u$.  
For example, in the STFT and the CWT, the window function and mother wavelet are free parameters. Instead of fixing the window function, we may consider a parametric space of window functions, parameterized by $u$, sharing the same linear volume discretization, and add the parameter $u$ as additional dimensions to phase space. For example, in the CWT we may choose as $u$ the spread of the mother wavelet.  
Integration of frames is the basis on which we construct the LTFT in Definition \ref{The localizing time-frequency continuous frame} below.
}

 





\subsection{Stochastic diffeomorphism operator and highly redundant phase vocoder}
\label{Stochastic diffeomorphism operator}

In this subsection, we study the signal processing pipeline when $T$ is a diffeomorphism operator, and propose a potential application in audio signal processing.

\subsubsection{Stochastic signal processing with diffeomorphism operators}

Let $f:G\rightarrow\cH$ be a bounded continuous frame, with bound $\norm{f_g}_{\cH}\leq C$, and suppose that the phase space $G$ is a Riemannian manifold. Let $\mathcal{R}\subset \cH$ be a class of signal, $\{V_M\}_M$ a discretization of $\mathcal{R}$, and $\{\eta_M\}_M$ a sequence of envelopes. Let $\e>0$, and suppose that $\big\{f,\mathcal{R},\{V_M\}_M,\{\eta_M\}_M\big\}$ is an $\e$-LVD of $f$.

Let $d:G\rightarrow G$ be a diffeomorphism (invertible smooth mapping with smooth inverse), with Jacobian $J_d\in L^{\infty}(G)$. 
Consider the diffeomorphism operator $T$, defined for any $F\in L^2(G)$ by
\begin{equation}
[TF](g) = F\big(d^{-1}(g)\big).
\label{Stoc_diff_T}
\end{equation}
Note that $\norm{T}_2= \norm{J_d}_{\infty}$.
Let $r:\CC\rightarrow\CC$ satisfy $\abs{r(x)}\leq E \abs{x}$. The signal processing pipeline based on the diffeomorphism $T$ is defined to be $\mathcal{P}_{f,T,r}S_f^{-1}s$ for the synthesis pipeline, and $S_f^{-1}\mathcal{P}_{f,T,r}s$ for the analysis pipeline.
The stochastic approximations of these pipelines are given on $s_M\in V_M$, up to the application of $S_f^{-1}$ from the right or from the left, by
	\begin{equation}
	[\mathcal{P}s_M]^{\eta_M, K}=   \frac{\norm{\eta_M}_1}{K} \sum_{k=1}^K r\big(V_f [s](g^k)\big)f_{d(g^k)}.
	\label{eq:StoDif1}
	\end{equation}
In (\ref{eq:StoDif1}), the points $\{g^k\}_{k=1}^K$ are sampled from the envelope $\eta_M$. This means that the points $\{d(g^k)\}_{k=1}^K$ are sampled from the envelope 
  $\psi_M(g)=\eta_M\big(d(g)\big)J_d(g)$ with \newline
  $\norm{\psi_M}_1=\norm{\eta_M\big(d(\cdot)\big)J_d(\cdot)}_1 = \norm{\eta_M}_1$. 
We can use either Theorem \ref{Main_Theorem12} or Theorem \ref{th:ful_lMC_SP} to bound the stochastic approximation error, and in either case we obtain an error of order $O(\frac{\sqrt{{\rm dim}(V_M)}}{\sqrt{K}})$.


\subsubsection{Integer time dilation phase vocoder}
\label{Phase vocoder example}



A time stretching phase vocoder is an audio effect that slows down an audio signal without dilating its frequency content. In the classical definition, $G$ is the time frequency plane, and $V_f$ is the STFT. Phase vocoder can be formulated as phase space signal processing in case the signal is dilated by an integer \cite[Section 7.4.3]{vocoder_book}. For an integer $\Delta$,
we consider the diffeomorphism operator $T$ with 
$d(g_1,g_2)=(\Delta g_1,g_2)$, 
and consider the nonlinearity $r$, defined by
\begin{equation}
    \label{eq:PhaseC}
    r(e^{i\theta}a)=e^{i\Delta\theta}a,
\end{equation} for $a\in\RR_+$ and $\theta\in\RR$. 
 The phase vocoder is defined to be
 $s\mapsto V_f^* T r\circ V_f[s]$. 
Note that since the STFT is a Parseval frame, there is no difference between analysis and synthesis signal processing.

Next, we replace the STFT frame in the phase vocoder with a highly redundant time-frequency representation, based on a 3D phase space.

\subsubsection{The localizing time-frequency transform}
\label{The localizing time-frequency transform}
Here, we construct an example redundant time frequency transform based on a combination of CWT atoms and STFT atoms.
The CWT is better than the STFT at isolating transient high frequency events, since middle to high frequency wavelet atoms have shorter time supports than LTFT atoms. On the other hand, low frequency events are smeared by the CWT, since low frequency wavelet atoms have large supports. We thus use STFT atoms to represent low frequencies, and CWT atoms to represent middle frequencies. High frequencies are represented again by STFT atoms with narrow time supports. This is done to potentially avoid false positive detection of very short transient events by very short wavelet atoms. 

We then add to this 2D time-frequency system a third axis that controls the number of oscillations in the CWT atoms. We motivate this as follows. Time-frequency atoms are subject to the uncertainty principle. The more accurately a time-frequency atom measures frequency, the less accurately it measures time. Different signal features call for a different balance between the time and the frequency measurement accuracy. In polyphonic audio signals we expect a range of such appropriate balances, which means that no choice of window is appropriate for all features. Hence, the addition of the \emph{number of oscillations} axis may be useful for representing a variety of features in polyphonic audio signals.

Consider a non-negative real valued window $h(t)$ supported in $[-1/2,1/2]$. For example, the \emph{Hann window} is defined to be $h(t) = \big(1+cos(2\pi t)\big)/2$, and zero for $t\notin[-1/2,1/2]$.
Consider a parameter $\tau$ that controls the \emph{number of oscillations} in the CWT atoms. We denote by $0<\tau_1<\tau_2$ the minimal and maximal number of oscillations of the wavelet atoms. The LTFT phase space is defined to be $G=\RR^2\times [\tau_1,\tau_2]$, where the measure $\mu_3$ on $[\tau_1,\tau_2]$ is any weighted Lebesgue measure with $\mu_3([\tau_1,\tau_2])=1$.  There are two transition frequencies in the LTFT, where the atoms change from STFT to CWT atoms and back. In general, we allow these transition frequencies $0<a_{\tau}<b_{\tau}<\infty$ to depend on $\tau$.

\begin{definition}[The localizing time-frequency continuous frame]
\label{The localizing time-frequency continuous frame}
Consider the above setting.
 The LTFT atoms are defined for $(x,\w,\tau)\in \RR^2\times [\tau_1,\tau_2]$, where $x$ represents time, $\w$ frequencies, and $\tau$ the number of wavelet oscillations, by
\begin{equation}
f_{x,\w,\tau}(t) = \left\{
\begin{array}{ccc}
	\sqrt{\frac{a_{\tau}}{\tau}}h\big(\frac{a_{\tau}}{\tau}(t-x)\big)e^{2\pi i \w (t-x)} & {\rm if} & \abs{\w}<a_{\tau} \\
	\sqrt{\frac{\w}{\tau}}h\big(\frac{\w}{\tau}(t-x)\big)e^{2\pi i \w (t-x)} & {\rm if} & a_{\tau}\leq\abs{\w}\leq b_{\tau} \\
	\sqrt{\frac{b}{\tau}}h\big(\frac{b_{\tau}}{\tau}(t-x)\big)e^{2\pi i \w (t-x)}  & {\rm if} &  b_{\tau}<\abs{\w}
\end{array}
\right.
\label{eq:LTFT_atom}
\end{equation}
\end{definition}

In the companion paper \cite{Ours_app} we prove that the LTFT is an LVD  continuous frame. This is natural in view of Subsection \ref{Integration of coherent state systems},  since, up to the low and high frequency truncation, the LTFT is based on integrating LVD wavelet transforms. 

\subsubsection{LTFT-based phase vocoder}

In this subsection we offer a toy example application of the LTFT, namely, integer time stretching phase vocoder based on the implementation of Subsection \ref{Phase vocoder example}.
The integer time dilation phase modification $r$ of (\ref{eq:PhaseC}) is said to preserve the \textit{horizontal phase coherence}, and deals well with slowly varying instantaneous frequencies \cite{Horizontal}. When the instantaneous frequency of a component in a signal changes rapidly, the phase modification $r$ is not sufficient, and methods for ``locking'' the phase to a frequency bin outside the horizontal line are used in modern implementations \cite{phasiness0,vocoder_imp}. Nevertheless, in this toy application we consider horizontal phase lock. One potential motivation for using this simplistic model is that phasiness artifacts\footnote{Phasiness \cite{phasiness0} is the audible artifact resulting from summing two time-frequency atoms with intersecting time and \textcolor{black}{frequency} supports, but with out of sync phases} may be alleviated by using CWT atoms, since CWT atoms have shorter time supports than STFT atoms.

An example implementation of stochastic LTFT phase vocoder  (\ref{eq:StoDif1}) is given in \newline \url{https://github.com/RonLevie/LTFT-Phase-Vocoder}.
In the companion paper \cite{Ours_app}, we prove that the total number of scalar operations in the phase vocoder LTFT method is $O(ZM + M\log(M))$, when the number of Monte Carlo samples is $K=ZM$.
In future work, we will study more modern phase vocoder implementations based on the LTFT, akin to \cite{New_vocoder0,New_vocoder1,ltfatnote050,Ottosen2017APV}.


\bibliographystyle{spmpsci}    
\bibliography{Ref_uncertainty3,RandNLA}

\appendix

\section{Hilbert space Berntein inequality}
\label{Hilbert space Berntein inequality}

The proof of Theorem \ref{H_Bern} is based on the finite dimensional counterpart from \cite[Theorem 2.6]{Berns}. There, the theorem is formulated for vectors in $\RR^n$, but can easily be extended to $\CC^n$. 

\begin{theorem}[Finite dimensional Bernstein inequality \cite{Berns}]
\label{C_Bern}
Let $\{v_k\}_{k=1}^K\subset \CC^d$ be a finite sequence of independent random vectors. Suppose that $\mathbb{E}(v_k)=0$ and $\norm{v_k}_{2}\leq B$ a.s. and assume that $\rho_K^2> \sum_{k=1}^K \mathbb{E}\norm{v_k}_{2}^2$. Then for all $0\leq t \leq \rho_K^2/B$,
\[P\left( \norm{\sum_{k=1}^K v_k}_{2}\geq t\right)\leq \exp\left(-\frac{t^2}{8\rho_K^2}+\frac{1}{4}\right).\]
\end{theorem}

In the following proof of Theorem \ref{H_Bern}, we use the fact that for weakly integrable random vectors $v:G\rightarrow \cH$ and bounded operators $T$ on $\cH$, we have
\begin{equation}
\label{eq:weakT}
    \int_G^{\rm w} T v(g) dg =  T\int_G^{\rm w}  v(g) dg .
\end{equation}

\begin{proof}[Proof of Theorem \ref{H_Bern}]
For a fixed $K$ we denote $\rho=\rho_K$.
\textcolor{black}{Let $\{w_l\}_{l=1}^{\infty}$ be an orthonormal basis in $\cH$, and $P_j$ be the orthogonal projection upon ${\rm span}\{w_l\}_{l=1}^j$.} 
Let us use Theorem \ref{C_Bern} on the random vectors $\{v^j_k\}_{k=1}^K= \{P_jv_k\}_{k=1}^K$, as vectors of $\CC^{d_j}$, for fixed $j$. By (\ref{eq:weakT}), we have
\[\mathbb{E}^{\rm w}(v_k^j)=\mathbb{E}^{\rm w}(P_jv_k)= P_j \mathbb{E}^{\rm w}(v_k)=\mathbb{E}(v_k)=0.\]
Next, by the fact that $P_j$ is a projection
\[\norm{v^j_k}_2 = \norm{P_j v_k}_{\cH} \leq \norm{v_k}_{\cH} \leq B.\]
Now, the pointwise bound $\norm{P_j v_k}_{\cH}^2 \leq \norm{v_k}_{\cH}^2$ carries to the integrals in the calculation of the expected values, so
\[\sum_{k=1}^K \mathbb{E}\norm{v_k^j}_2^2 =\sum_{k=1}^K \mathbb{E}\norm{P_j v_k}_{\cH}^2 \leq  \sum_{k=1}^K \mathbb{E}\norm{v_k}_{\cH}^2 \leq \rho^2.\]
Thus, Theorem \ref{C_Bern} gives
\begin{equation}
\forall \ 0\leq t \leq \rho^2/B\ .\quad P\Big(\norm{\sum_{k=1}^K v_k^j}_2\geq t \Big) \leq \exp\Big(-\frac{t^2}{8\rho^2}+\frac{1}{4}\Big).
\label{eq:in_proof_bern1}
\end{equation}

Next, we show that (\ref{eq:in_proof_bern1}) carries also in the limit as $j\rightarrow\infty$.
Consider the following functions in the probability space $G^K$:
the characteristic function $\chi$ of the set
\[\Big\{{\bf g}\ \big|\ \norm{\sum_{k=1}^K v_k({\bf g})}_{\cH}>t\Big\},\]
and  characteristic function $\chi_j$ of
\[\Big\{{\bf g}\ \big|\ \norm{\sum_{k=1}^K v^j_k({\bf g})}_2>t\Big\}.\]
By the fact that projections reduce norms, $\chi_j({\bf g})\leq \chi({\bf g})$ for every ${\bf g}\in G_0^K$. Moreover, $\chi_j$ is a pointwise monotone sequence of measurable functions. By the strong convergence of the projections $P_j$ to $I$, we have
\[\forall {\bf g}\in G_0^K\ . \quad \lim_{j\rightarrow\infty}{\chi_j({\bf g})} = \chi({\bf g}).\]
This is shown as follows. Let ${\bf g}=(g_1,\ldots,g_K)$ be a fixed point. If $\chi({\bf g})=0$ then it is trivial to see $\lim_{j\rightarrow\infty}{\chi_j({\bf g})} = \chi({\bf g})$. Otherwise, for every $\e>0$
there is a big enough $J_{\e}\in\NN$ such that for every  $j>J_{\e}$ we have
\[\abs{\ \norm{\sum_{k=1}^K v_k(g_k)}_{\cH} - \norm{\sum_{k=1}^K v_k^j(g_k) }_2\ }<\e.\]
Since $\chi({\bf g})=1$, we have
\[\norm{\sum_{k=1}^K v_k(g_k)}_{\cH}=r>t.\]
Therefore, for $\e<0.5(r-t)$, and any $j>J_{\e}$
\[ t<r-\e<\norm{\sum_{k=1}^K v_k^j}_2<r+\e\]
so
 $\chi_j({\bf g})=1$, 
which proves that $\lim_{j\rightarrow\infty}{\chi_j({\bf g})} = \chi({\bf g})$. \textcolor{black}{Lastly, (\ref{eq:H_bern}) follows from  Beppo Levi's monotone convergence theorem.} 


\end{proof}

\section{Pseudo inverse of frame analysis and synthesis operators}
\label{Pseudo inverse of frame analysis and synthesis operators}

For an injective linear operator with close range $B:\mathcal{V}\rightarrow \mathcal{W}$ between the Hilbert spaces $\mathcal{V}$ and $\mathcal{W}$, we define the pseudo inverse
\[B^+:\mathcal{W}\rightarrow \mathcal{V}, \quad B^+ = \big(B\big|_{\mathcal{V}\rightarrow B\mathcal{V}}\big)^{-1}R_{B\mathcal{V}},\]
where $R_{B\mathcal{V}}:\mathcal{W}\rightarrow B\mathcal{V}$ is the surjective operator given by the orthogonal projection from $\mathcal{W}$ to $B\mathcal{V}$ and restriction of the image space to the range $B\mathcal{V}$,  
 and $B\big|_{\mathcal{V}\rightarrow B\mathcal{V}}$ is the restriction of the image space of $B$ to its range $B\mathcal{V}$, in which it is invertible.
 Note that $R_{B\mathcal{V}}^*$ is the operator that takes a vector from $B\mathcal{V}$ and canonically embeds it in $\mathcal{W}$, and $P_{B\mathcal{V}}=R_{B\mathcal{V}}^*R_{B\mathcal{V}}:\mathcal{W}\rightarrow\mathcal{W}$ is the orthogonal projection upon $B\mathcal{V}$. 
In the following we list basic properties of $V_f$ and $V_f^+$.

\begin{lemma}
\label{Lemm_PI}
Let $f:G\rightarrow\cH$ be a continuous frame. Then the following properties hold.
\begin{enumerate}
  \item 
	\label{Lemm_PI:0}
	$V_f^+ V_f = I$.
	\item 
	\label{Lemm_PI:1}
	$V_f V_f^+ = P_{V_f[\cH]}$.
	\item
	\label{Lemm_PI:10}
	$V_f^+ P_{V_f[\cH]} = V_f^+$.
	\item 
	\label{Lemm_PI:2}
	The operator $V_f^*$ is the pseudo inverse of $V_f^{+*}$, and $V_f^{+*}[\cH]=V_f[\cH]$.
	\item 
	\label{Lemm_PI:3}
	$(V_f^*V_f)^{-1}= V_f^+V_f^{+*}$.
\end{enumerate}
\end{lemma}

\end{document}